\documentclass[11pt,a4paper,reqno]{amsart}

\usepackage{amsmath, amssymb,xspace}

\usepackage[hyperindex,breaklinks,colorlinks=true,linkcolor=black,anchorcolor=black,citecolor=black,filecolor=black,menucolor=black,runcolor=black,urlcolor=blue]{hyperref}

\usepackage{latexsym,amssymb,amsmath}

\usepackage{enumerate}
\usepackage{amsmath}

\usepackage{amsfonts}
\usepackage{amssymb}
\usepackage{eucal}

\usepackage{amssymb}

   \newcommand {\eop}{\hfill $\square$}

\newcommand{\ria}{\rightarrow}
\newcommand{\Qbinary}{\mathbb{Q}_2}

\newcommand{\Halt}{{\ES'}}

\newcommand{\HH}{\mathcal{K}}

\newcommand{\UM}{\mathbb{U}}

\newcommand{\exo}[1]{\exists #1 \, }
\newcommand{\fao}[1]{\forall #1 \, }

\newcommand{\lwtt}{\le_{\mathrm{wtt}}}
\newcommand{\ltt}{\le_{\mathrm{tt}}}

\newcommand{\verif}{\n {\it Verification.\ }}

 \newcommand{\vsps}{\vspace{3pt}}

\newcommand{\mmbox}[1]{\ \mbox{#1}\ }

\newcommand{\ttext}[1]{\ \text{#1}\ }

\newcommand{\SI}[1]{\Sigma^0_{#1}}
\newcommand{\PI}[1]{\Pi^0_{#1}}

\newcommand{\PPI}{\PI{1}}

\newcommand{\PP}{\mathcal P}

\renewcommand{\P}{\mathcal P}
\newcommand{\Q}{\mathcal Q}

\newcommand{\QI}{\Pi^1_1}

\newcommand{\Kuc}{Ku{\v c}era}
\newcommand{\ML}{Martin-L{\"o}f}

\newcommand{\rapf}{$\RA:$\ }
\newcommand{\lapf}{$\LA:$\ }

\newcommand{\strcantor}{2^{<\omega}}
\newcommand{\seqcantor}{2^{ \omega}}
\newcommand{\cantor}{\seqcantor}

\newcommand{\init}{\mbox{\rm \textsf{init}}}

\newcommand{\NN}{{\mathbb{N}}}  % Symbols for sets of  numbers

\newcommand{\QQ}{{\mathbb{Q}}}

\newcommand{\DII}{\Delta^0_2}

\newcommand{\lep}{\le^+}

\newcommand{\ex}{\exists}
\newcommand{\fa}{\forall}

\newcommand{\LLR}{\ \Leftrightarrow \ }
\newcommand{\RA}{\Rightarrow}
\newcommand{\RRA}{\ \Rightarrow\ }
\newcommand{\LA}{\Leftarrow}
\newcommand{\DA}{\downarrow}
\newcommand{\UA}{\uparrow}
\newcommand{\n}{\noindent}
\newcommand{\wt}{\widetilde}

\newcommand{\sub}{\subseteq}

\newcommand{\ul}{\underline}

\renewcommand{\land}{\&}

\newcommand{\ES}{\emptyset}

\renewcommand{\hat}{\widehat}

\newcommand{\dom}{\mbox{dom}}
\newcommand{\lland}{\ \land \ }

\newcommand{\itone}{\item[(i)]}
\newcommand{\ittwo}{\item[(ii)]}
\newcommand{\itthree}{\item[(iii)]}

\newcommand{\la}{\langle}
\newcommand{\ra}{\rangle}

\newcommand{\bi}{\begin{itemize}}
\newcommand{\ei}{\end{itemize}}
\newcommand{\bc}{\begin{center}}
\newcommand{\ec}{\end{center}}

\newcommand{\sss}{\sigma}
\newcommand{\aaa}{\alpha}

\newcommand{\sssl}{|\sigma|}
\newcommand{\aaal}{|\alpha|}

\newcommand{\leT}{\le_T}
\newcommand{\geT}{\ge_T}

\newcommand{\leb}{\mathbf{\lambda}}

\newcommand{\Om}{\Omega}
\newcommand{\twoset}{\{0,1\}}
\newcommand{\PF}{prefix-free }

\newcommand{\sN}[1]{_{#1\in \NN}}

\newcommand{\seq}[1]{{\langle{#1}\rangle}} 

\newcommand{\estring}{\emptyset}

\newcommand{\Opcl}[1]{[#1]^\prec}
\newcommand{\tp}[1]{2^{#1}}

\newcommand\+[1]{\mathcal{#1}}

\newcommand{\uhr}[1]{\!  \upharpoonright_{#1}}

\newcommand{\Cyl}[1]{\ensuremath{[ \! [ {#1} ] \! ]}}

\newtheorem{theorem}{Theorem}[section]
\newtheorem{thm}[theorem]{Theorem}

\newtheorem{definability lemma}[theorem]{Definability Lemma}

\newtheorem{fact}[theorem]{Fact}
\newtheorem{example}[theorem]{Example}
\newtheorem{proposition}[theorem]{Proposition}
\newtheorem{prop}[theorem]{Proposition}
\newtheorem{claim}[theorem]{Claim}

\newtheorem{definition}[theorem]{Definition}
\newtheorem{deff}[theorem]{Definition}

\newtheorem{convention}[theorem]{Convention}

\newtheorem{cor}[theorem]{Corollary}

\newtheorem{remark}[theorem]{Remark}
\renewcommand{\SS}{\mathcal{S}}

\newcommand{\cc}{\mathbf{c}}
\newcommand{\dd}{ \mathbf d}

\renewcommand{\k}{\ensuremath{\text{\bfseries\itshape k}\, }}

\newcommand{\vsp}{\vspace{6pt}}

\newcommand{\cost}{\cc_{\HH}}
\newcommand{\ulcost}{\ul \cc_{\HH}}

\begin{document}
\title{Calculus of Cost Functions}
%\title{Subclasses of the $K$-trivial sets}
\author{ 
Andr\'e Nies}



\thanks{
Research partially supported by the Marsden Fund of New
Zealand, grant no.\ 08-UOA-187,  and by the Hausdorff Institute of Mathematics, Bonn.}

\begin{abstract} Cost functions provide  a framework for  constructions of   sets Turing below the halting problem that are close to computable.  We carry out  a systematic study of cost functions. We  relate their algebraic properties to their expressive strength. We show that the  class of additive cost functions   describes  the $K$-trivial sets. We prove a cost function basis theorem, and give a general construction for   building   computably enumerable sets that are  close to being Turing complete.
\end{abstract}
\maketitle
\tableofcontents
\section{Introduction}  
 In the time period from 1986 to 2003, several  constructions of  computably enumerable (c.e.) sets    appeared. They turned out to be closely related.
\bi \item[(a)] Given a  \ML\ random (ML-random for short)  $\DII$ set $Y$,

\n  \Kuc~\cite{Kucera:86}  built a c.e.\ incomputable set  $A\leT Y$.  His construction is interesting because in the  case that  $Y <_T \Halt$, it provides a   c.e.\ set~$A$ such that $\ES <_T A <_T \Halt$,  without using injury to requirements as in the traditional proofs. ($\Halt$ denotes the halting problem.)

\item[(b)] \Kuc\ and Terwijn~\cite{Kucera.Terwijn:99} built a c.e.\ incomputable set  $A$ that is low for ML-randomness: every ML-random set  is already ML-random relative to $A$. %This means that $A$ cannot ``derandomize'' any  ML-random sets.

\item[(c)] $A$ is called $K$-trivial if  $K(A\uhr n) \le K(n)+ O(1)$, where $K$ denotes prefix-free descriptive string complexity. This means that the initial segment complexity of $A$ grows as slowly as that of a computable set. Downey et al.\ \cite{Downey.Hirschfeldt.ea:03} gave a very short  construction  (almost a ``definition'')   of a  c.e., but incomputable $K$-trivial  set.  
\ei

 The sets in (a) and (b) enjoy  a so-called lowness property, which says that the set is  very   close to computable.   Such properties can be classified according to various paradigms introduced in \cite{Nies:ICM, Greenberg.Hirschfeldt.ea:12}.   The set in (a) obeys  the  \emph{Turing-below-many}  paradigm which  says that $A$ is close to being  computable because it  is easy for an oracle set to compute it.  
A frequent alternative is  the  \emph{weak-as-an-oracle}  paradigm:  $A$ is  weak in a specific sense when used as an oracle set in a Turing machine computation.  An example is the oracle set in (b), which is so weak that it useless as an extra computational device when testing for ML-randomness.  On the other hand, $K$-triviality in (c) is a property stating that the set is far from random:  by the Schnorr-Levin Theorem, for a random set  $Z$ the initial segment complexity grows fast in that $K(Z \uhr n) \ge n - O(1)$.  For   background on the   properties     in (a)-(c) see   \cite{Downey.Hirschfeldt:book} and \cite[Ch.\ 5]{Nies:book}.\footnote{We note that the result (c) has a complicated history. Solovay \cite{Solovay:75} built a $\DII$  incomputable set  $A$ that is $K$-trivial. Constructing a  c.e.\ example of such a set  was attempted  in various sources such as \cite{Calude.Grozea:96}, and unpublished work of Kummer.}

A central point for  starting  our investigations is  the fact  that the   constructions in (a)--(c)  look very    similar. In hindsight this is not surprising: the classes of sets  implicit  in (a)-(c) coincide!  Let us discuss why. 

\vsps 

\n {\it (b) coincides with  (c):}  Nies \cite{Nies:AM}, with some assistance by Hirschfeldt,  showed that  lowness for ML-randomness is the same as   $K$-triviality.  For this he introduced a method now known as the ``golden run''. 

\vsps
\n {\it (a) coincides with  (b):}    The construction in  (a) is only interesting if $Y \not \ge_T \Halt$.  Hirschfeldt, Nies and Stephan~\cite{Hirschfeldt.Nies.ea:07} proved that if $A$ is a  c.e.\ set   such that $A\leT Y$ for some ML-random set  $Y \not \ge_T \Halt$, then $A$ is $K$-trivial, confirming the intuition   sets of the type built by \Kuc\  are close to computable.  They   asked whether, conversely, for every $K$-trivial set $A$ there is a ML-random set  $Y\ge_T A$ with $Y \not \ge_T \Halt$. This question became known as the  ML-covering problem. Recently the  question was solved in the affirmative by combining the work of  seven authors in two separate papers. In fact,   there is a single ML-random $\DII$ set  $Y \not \ge_T \Halt$ that is Turing above all the $K$-trivials. A summary is given in~\cite{Bienvenu.Day.ea:14}. %Thus the types of sets obtained in (a) and (b) are also the same.

\vsps

  The common idea for these  constructions  is to ensure lowness of $A$ dynamically, by restricting the overall manner in which  numbers can be enumerated into $A$. This third lowness paradigm has been called  \emph{inertness}  in \cite{Nies:ICM}: a set   $A$ is close to computable because  it is  computably 
approximable with a small number  of    changes.  

The idea is implemented as follows. The enumeration of a number $x$ into $A$ at a stage~$s$ bears a cost $\cc(x,s)$, a non-negative rational that can be computed from $x$ and~$s$. We have to enumerate $A$ in such a way that the sum of all costs is finite.  A construction of this type will be called a  {\em cost function construction}. 

 If we enumerate at a stage more than one number into $A$, only the cost for enumerating the least number is  charged. So, we can  reduce  cost by enumerating $A$ in ``chunks''.

\subsection{Background on cost functions} The general theory of cost functions  began in~\cite[Section 5.3]{Nies:book}. It was further developed in  \cite{Greenberg.Nies:11,Greenberg.Hirschfeldt.ea:12,Diamondstone.Greenberg.ea:nd}. We use the language of~\cite[Section 5.3]{Nies:book} which  already  allows for  the  constructions of $\DII$ sets. The language is  enriched by some   notation from~\cite{Diamondstone.Greenberg.ea:nd}. We will see that most examples of cost functions are based on randomness-related concepts.

\begin{deff} \label{def:c.f.}   {\rm    A \emph{cost function} is a computable function \bc $\cc:
\NN \times \NN \ria \{x \in \QQ \colon \,  x \ge 0\}$. \ec } %such that $ \cc(x,s)=0 $ for $x \ge s$. 
\end{deff}
\n Recall  that a  \emph{computable approximation}  is a computable sequence of finite sets $\seq{ A_s}_{s\in \NN}$ such that $\lim_s A_s(x)$ exists for each $x$. 
\begin{deff} \label{df:obey cf}  
{\rm 
\n (i). Given  a  computable approximation $\seq{ A_s}_{s\in \NN}$
 and a  cost function~$\cc$, for $s>0$  we let \bc $\cc_s(A_s) = \cc(x,s) $ where $ x< s  \lland x \mmbox{\rm is
least s.t.}   A_{s-1} ( x) \neq  A_{s} (x) $; \ec 
if there is no such $x$ we let $\cc_s(A_s) =0$. This is the cost of  changing $A_{s-1}$ to~$A_s$.
  We let 
\begin{equation*}   \cc \seq{A_s}\sN s = \sum_{s>0} \cc_s(A_s) \end{equation*}
 be the total cost of all the $A$-changes. We will often   write  $\cc\seq {A_s}$ as a shorthand for $\cc \seq{A_s}\sN s$.
 
\vsps

\n (ii) We say  that  $\seq{A_s}_{s\in \NN}$ \emph{obeys} $\cc$   if $\cc \seq{A_s}$ is finite. We denote this by
\bc    $\seq{A_s} \models \cc$. \ec

\vsps
\n (iii) We  say that a $\DII$   set~$A$ \emph{obeys} $\cc$, and write  $A \models \cc$,  if some  computable approximation of $A$ obeys $\cc$.    }
\end{deff}

A cost function $\cc$ acts like a global restraint, which  is successful if  the condition $ \cc  {\seq{A_s}}   < \infty$  holds.  \Kuc's  construction mentioned in  (a) above needs to be recast in order  to be viewed as a cost-function construction \cite{Greenberg.Nies:11,Nies:book}. In contrast,   (b)  and (c)   can   be directly seen as cost function constructions. 
In  each of (a)--(c) above, one   defines a cost function $\cc$ such that any set $A$ obeying $\cc$ has the lowness property in question.  For, if $A \models \cc$, then  one can enumerate an auxiliary object that  has in some sense a bounded  weight. 

In (a), this object is a Solovay test that  accumulates the errors in an attempted computation of~$A$ with oracle $Y$. Since  $Y$ passes this test,  $Y$  computes $A$. 

In (b), one is given a $\Sigma^0_1(A)$ class $\+ V \sub \cantor$ such that  the  uniform measure  $\leb \+ V$ is less than~$1$, and the complement of $\+ V$ consists only of ML-randoms. Using that $A$ obeys $\cc$,  one  builds  a $\Sigma^0_1$ class $\+ S \sub \cantor$ containing $\+ V$ such that still $\leb \+ S <1$. This implies that $A$ is low for ML-randomness. 

In (c) one builds   a bounded request set (i.e., Kraft-Chaitin set) which shows that $A$ is $K$-trivial. 

The cost function in (b) is adaptive in the sense that $\cc(x,s) $ depends on $A_{s-1}$. In contrast,  the   cost functions in (a) and (c)  can be defined  in advance, independently  of the computable approximation of the set  $A$ that is   built.

The main existence theorem, which we recall as Theorem~\ref{thm:cfconstr} below, states  that for any cost function $\cc$ with the limit condition $\lim_x \liminf_s \cc(x,s) = 0$, there is an incomputable c.e.\ set $A$ obeying $\cc$. The   cost functions in (a)-(c) all have  the limit condition. Thus, by the existence theorem,  there is an incomputable c.e.\  set $A$ with  the required lowness property.

Besides providing a unifying picture of these constructions, cost functions have many other applications. We discuss some of them.

Weak 2-randomness is a notion stronger than ML-randomness: a set $Z$ is weakly 2-random if $Z$ is in no $\PI 2$ null class. In 2006,  Hirschfeldt and Miller   gave a characterization of this notion: a ML-random is weakly 2-random if and only if it forms a minimal  pair with $\Halt$. The implication from left to right is straightforward. The converse direction relies on a cost function related  to the one for \Kuc's result (a) above.  (For detail see e.g.\ \cite[Thm.\ 5.3.6]{Nies:book}.) Their result  can be seen  as an instance of    the randomness enhancement principle~\cite{Nies:ICM}: the ML-random sets get more random as they lose computational complexity.

The author \cite{Nies:AM} proved  that  the  single cost function  $\cost$ introduced in  \cite{Downey.Hirschfeldt.ea:03}    (see Subsection~\ref{ss:standard_costfunction} below) characterises the $K$-trivials. As a corollary, he showed that every $K$-trivial set $A$ is truth-table below  a c.e.\ $K$-trivial $D$. The proof of this corollary uses the  general framework  of change sets spelled out in Proposition~\ref{prop:c.e. ub cf} below. While this  is still  the only known proof yielding $A\ltt D$,  Bienvenu  et al.\
 \cite{Bienvenu.Downey.ea:15} have recently given an alternative proof using Solovay functions in order to  obtain  the weaker reduction  $A \leT D$.

 In model theory, one asks whether a class of structures can be described by  a first order theory. Analogously, we ask whether an   ideal of the Turing degrees below $\mathbf 0'$ is given by obedience to all cost functions of  an appropriate type. For instance, the $K$-trivials are axiomatized by $\cost$. 

 Call a cost function $\cc$ {\em benign} if from $n$ one can compute  a bound on the number of disjoint intervals $[x,s)$ such that  $\cc(x,s) \ge \tp{-n}$. Figueira et al.~\cite{Figueira.ea:08} introduced  the property of being strongly jump traceable (s.j.t.), which    is an extreme lowness property of an oracle $A$, even stronger than being low for  $K$. Roughly speaking, $A$ is s.j.t.\ if the jump $J^A(x)$ is in $T_x$  whenever it is defined, where $\seq{T_x}$ is a uniformly c.e.\ sequence of sets such that  any given order function  bounds  the  size of  almost all the $T_x$.    Greenberg and Nies~\cite{Greenberg.Nies:11} showed that the  class of benign cost functions axiomatizes the c.e.\ strongly jump traceable  sets.

Greenberg et al.\ \cite{Greenberg.Hirschfeldt.ea:12} used cost functions to show that each strongly jump-traceable  c.e.\ set is Turing below each $\omega$-c.e. ML-random set. As a main result, they also obtained  the converse. In fact they showed that any set that is below each superlow ML-random set is s.j.t.  

 The question remained  whether  a  general s.j.t.\ set  is Turing below each $\omega$-c.e.\  ML-random set.  Diamondstone et al.\ \cite{Diamondstone.Greenberg.ea:nd} showed that each s.j.t.\ set $A$ is Turing below a c.e., s.j.t.\ set $D$. To do so,  as  a main technical result they provided  a benign cost function  $\cc$ such that each set $A$ obeying $\cc$ is Turing below a c.e.\ set $D$ which obeys every cost function that $A$ obeys. In particular, if $A$ is s.j.t., then $A \models \cc$, so the c.e.\ cover $D$ exists and is also s.j.t.\ by the above-mentioned result of Greenberg and Nies~\cite{Greenberg.Nies:11}. This gives an affirmative answer to the question.
Note that this answer  is analogous  to the     result~\cite{Bienvenu.Day.ea:14} that every $K$-trivial is below an incomplete random.

\subsection{Overview of our results}
The main purpose of the paper is a systematic  study  of cost functions and the  sets obeying them. We are guided by the above-mentioned  analogy from first-order model theory: cost functions are like sentences, sets are like models, and obedience is like  satisfaction.  So far this analogy has been developed only for cost functions that are monotonic (that is, non-increasing in the first component, non-decreasing in the stage component). In Section~\ref{s:look_ahead} we show that the conjunction of two  monotonic cost functions is given by their sum, and implication $\cc \to \dd$ is equivalent to  $\ul \dd = O(\ul \cc)$ where $\ul \cc(x) = \sup_s \cc(x,s)$ is the limit function.

In Section~\ref{s:additive cf} we show that  a natural  class of cost functions introduced  in Nies~\cite{Nies:ICM}       characterizes the $K$-trivial  sets: a cost function $\cc$ is additive if $\cc(x,y)+ \cc(y,z) = \cc(x,z)$ for all $x< y< z$. We show that such a cost function is  given by  an enumeration of a left-c.e.\ real, and that implication corresponds to Solovay reducibility on left-c.e.\ reals.  Additive cost functions have been used prominently in the solution of the ML-covering problem~\cite{Bienvenu.Day.ea:14}. The fact that a given $K$-trivial $A$ obeys every additive cost function is used to show that $A \leT Y$ for the Turing incomplete ML-random set constructed by Day and Miller~\cite{Day.Miller:14}.

Section~\ref{s:random_low_K-triv} contains some more applications of cost functions to the study of computational lowness and $K$-triviality. For instance, strengthening the result in \cite{Greenberg.Hirschfeldt.ea:12} mentioned above, we show that each c.e., s.j.t.\ set is below any complex $\omega$-c.e.\  set $Y$,  namely, a set $Y$ such that  there is an order function $g$ with $g(n ) \lep K(Y \uhr n)$ for each $n$. In addition,  the use of the reduction is bounded by the identity. Thus, the full  ML-randomness assumed in \cite{Greenberg.Hirschfeldt.ea:12} was too strong a hypothesis. We also discuss the relationship of cost functions and a weakening of $K$-triviality.  

In the remaining part of the paper   we obtain two    existence theorems.  
  Section~\ref{s:costf_basis_theorem} 
 shows that  given   an arbitrary  monotonic  cost function $\cc$, any  nonempty $\PPI$ class contains a   $\DII$ set  $Y$ that is so low that   each c.e.\ set $A \leT Y$  obeys~$\cc$.   In Section~\ref{ss:dual} we relativize a cost function $\cc$ to an oracle  set $Z$, and show that there is a c.e.\ set $D$ such that $\Halt$ obeys $\cc^D$ relative to $D$. This much harder ``dual'' cost function construction can be used to build incomplete c.e.\ sets that are very close to computing $\Halt$. For instance, if $\cc$ is the cost function $\cost$ for $K$-triviality, then $D$ is LR-complete.

%\footnote{This  research was  appropriate for  the   time of economic hardship the world passed through during 2008-2010. Only    cheap sets were built.}

\section{Basics} 
We provide formal background, basic facts and examples relating to the discussion above. We introduce classes of cost functions: monotonic, and proper cost functions. We  formally define the limit condition, and give a proof of the existence theorem. 
\subsection{Some easy facts on cost functions} 

\mbox{}

\begin{deff} \label{def:c.f.-2}    
 {\rm  We say that a cost function~$\cc$ 
   is \emph{nonincreasing  in the main argument}    if   \bc $\fa x, s \, [  \cc(x+1, s) \le  \cc(x, s)]$.   \ec  We say that~$\cc$ 
   is \emph{nondecreasing in the stage}     if  $\cc(x, s) = 0 $ for $x >  s$ and   \bc $\fa x, s \, [  \cc(x, s) \le  \cc(x, s+1)] $. \ec 
     If $\cc$ has both properties we say that $\cc$ is \emph{monotonic}.  This means that the cost $ \cc(x,s)$ does not decrease when we enlarge the interval $[x,s]$.    }
\end{deff}

\begin{fact} \label{easy1} Suppose $A\models \cc$. Then for each $\epsilon >0$ there is a computable approximation $\seq{A_s}\sN s$ of $A$ such that $\cc  {\seq{A_s}\sN s}  < \epsilon$. \eop  \end{fact}
\begin{proof} Suppose $\seq{\widehat A_s}\sN s \models \cc$. Given    $x_0$  consider the modified computable approximation $\seq{ A_s}\sN s$ of $A$   that  always outputs the final value $A(x)$  for  each   $x\le x_0$. That is, $A_s(x) =A(x)$ for $x \le x_0$, and  $A_s(x) =\widehat A_s(x)$ for $x > x_0$.
 	 Choosing $x_0$ sufficiently large,   we can ensure $\cc \seq A_s < \epsilon$.   
\end{proof}
%, in the definition of obedience \ref{df:obey cf},

\begin{definition} \label{def:proper cf} {\rm Suppose that  a  cost function $ \cc(x,t)$ is non-increasing in  the main argument~$x$. We say that  $\cc$ is  {\it proper}  if    $\fao x \ex t \,   \cc(x,t) > 0$.} \end{definition}
 If   a cost function that is non-increasing in  the main argument  is not proper, then  every  $\DII$ set obeys~$\cc$. Usually we will  henceforth assume that   a cost function $\cc$ is proper. Here is an example how being proper helps. 
\begin{fact} Suppose that $\cc$ is a proper cost function   and  $S= \cc {\seq{A_s} } < \infty  $ is     a  computable  real. Then $A$ is computable. \end{fact}
\begin{proof} Given an input  $x$, compute  a stage  $t$ such that  $\delta =  \cc(x,t)>0$ and  $S-  \cc {\seq{A_s}_{s \le t} }  < \delta$. Then $A(x) = A_t(x)$.\end{proof}
A  \emph{computable enumeration}  is a computable approximation $\seq{ B_s}_{s\in \NN}$
 such that $B_s \sub B_{s+1}$ for each $s$. 
 \begin{fact}  \label{fact:fukd} Suppose $\cc$ is a monotonic cost function and $A \models \cc$ for a c.e.\ set~$A$. Then there is a computable \emph{enumeration}  $\seq{\wt A_s}$    that  obeys $\cc$.  \end{fact}

\begin{proof} Suppose $\seq{A_s}  \models  \cc$ for a computable approximation $\seq{A_s}$ of $A$. Let $\seq{B_t}$ be a computable enumeration of $A$. Define $\seq{\wt A_s} $ as follows.  Let $\wt A_0(x)= 0$; for $s> 0$ let  $\wt A_s (x)
= \wt A_{s-1}(x)$ if  $ \wt A_{s-1}(x)=1$; otherwise let  $\wt A_{s}(x)= A_t(x)$ where $t \ge s$ is least such that  $A_t(x)= B_t(x)$. 

Clearly $\seq{\wt A_s}$ is a computable enumeration of $A$. If  $ \wt A_s(x) \neq \wt A_{s-1}(x)$ then  $A_{s-1}(x)= 0 $ and $A_s(x)=1$. Therefore $\cc  {\seq{\wt A_s}}   \le \cc  {\seq{A_s}}   < \infty  $. \end{proof}

\subsection{The limit condition and the existence theorem}

\mbox{}

\n For a cost function $\cc$,  let 
 \begin{equation}  \ul \cc (x) = \liminf_s  \cc(x,s). \end{equation}
 
\begin{deff} \label{df:lc} We say that a cost function~$\cc$  satisfies the {limit condition}  if $\lim_x
\,   \ul \cc(x) = 0$.  That is,  for each~$e$,   for almost every~$x$ we have \bc $\ex^\infty  s  \, [  \cc(x,s) \le  \tp{-e}]$. \ec
\end{deff}

In previous works such as \cite{Nies:book}, the limit condition was defined in terms of   $\sup_s \cc(x,s)$, rather than    $\liminf_s  \cc(x,s)$. The  cost functions previously considered were usually nondecreasing in the stage component, in which case $ \sup_s \cc(x,s)  = \liminf_s  \cc(x,s)$ and hence the two versions of the limit condition are equivalent.  Note that the limit condition is a $\PI 3$ condition on  cost functions that are nondecreasing in the stage, and $\PI 4$ in general.
 
The basic  existence theorem says that a  cost function with the limit condition has a c.e., incomputable model.  This was proved by various  authors  for    particular cost functions. The following version of the proof  appeared in  \cite{Downey.Hirschfeldt.ea:03} for the  particular  cost function~$\cost$ defined in Subsection~\ref{ss:standard_costfunction} below, and then in full generality in 
 \cite[Thm 5.3.10]{Nies:book}. % Interestingly, no monotonicity hypothesis on the cost function  is required. 

\begin{thm} \label{thm:cfconstr} Let~$\cc$ be a cost function
with  the limit condition. 

\bi \item[(i)] There is a simple set~$A$ such that~$A\models \cc$. Moreover,     $A$ can be  obtained uniformly in  (a computable index for)~$\cc$. 
\item[(ii)] If $\cc$ is nondecreasing in the stage component, then we can make $A$ promptly simple. \ei
 \end{thm}

\begin{proof} (i) We meet  the usual     simplicity requirements
	  \bc $S_e$: \ $\# W_e  =\infty \RRA W_e \cap A \neq \ES.   $ \ec
 To do so,  we define a
computable enumeration $\seq{A_s}\sN{s}$ as follows.
 {Let $A_0 = \ES$. At stage $s>0$, for each $e < s$, if
$S_e$ has not been met so far  and there is $x \ge 2e$ such that $x
\in W_{e, s} $ and  $ \cc(x,s) \le 2^{-e}$,
  put~$x$ into $A_s$. Declare $S_e$ met.}
  
  \vsps

  To see  that $\seq{A_s}\sN{s} $ obeys~$ \cc$,   note that at most one number is
put  into~$A$ for the sake of each requirement. Thus $\cc {\seq{A_s}}   \le \sum_e \tp{-e} =2$.

 If $W_e$ is infinite,  then   there is an $x\ge 2e$ and $s>x$  such that  $x \in W_{e,s}$
and  $ \cc(x,s) \le 2^{-e}$,  because~$\cc$ satisfies the limit condition. So we meet $S_e$.
Clearly the construction of~$A$ is uniform in an index for  the computable function~$\cc$.

\vsps

\n (ii) Now  we meet  the     prompt simplicity requirements  \bc $PS_e$: \ $\# W_e  =\infty \RRA \exo s \exo x  [ x \in W_{e,s  } - W_{e,s-1}    \lland x \in A_{s}].   $ \ec

 {Let $A_0 = \ES$. At stage $s>0$, for each $e < s$, if
$PS_e$ has not been met so far  and there is $x \ge 2e$ such that $x
\in W_{e, s} - W_{e,s-1} $ and  $ \cc(x,s) \le 2^{-e}$,
  put~$x$ into $A_s$. Declare $PS_e$ met.}

 If $W_e$ is infinite,    there is an $x\ge 2e$ in $W_e$
such that $ \cc(x,s) \le 2^{-e}$ for all $s>x$,  because~$\cc$ satisfies the limit condition and is nondecreasing in the stage component.  We enumerate such an
$x$ into~$A$ at the stage $s>x$ when~$x$ appears in $W_e$, if
$PS_e$ has not been met yet by stage~$s$. Thus~$A$ is promptly simple.
     \end{proof}
Theorem~\ref{thm:cfconstr}(i)   was strengthened  in \cite[Thm 5.3.22]{Nies:book}. As before let $\cc$ be a cost function
with  the limit condition. Then for each low c.e.\ set~$B$,
 there is a c.e.\ set~$A$ obeying~$\cc$ such that $A\not \leT B$.   The proof of  \cite[Thm 5.3.22]{Nies:book}  is for the case of the stronger version of the limit condition $\lim_x \sup_s \cc(x,s)=0$, but in fact works for the     version given above.

 The assumption that $B$ be c.e.\ is necessary:  there is a low set Turing above all the $K$-trivial sets by~\cite{Kucera.Slaman:09}, and the $K$-trivial sets can be  characterized as the sets obeying the  cost function $\cost$ of Subsection~\ref{ss:standard_costfunction} below.

 The following fact implies the  converse of Theorem~\ref{thm:cfconstr} in the monotonic case.
\begin{fact} \label{ex:limit cond necc}    
\n Let $\cc$ be a  monotonic cost function. If a computable approximation  $\seq{A_s}_{s\in
\NN}$  of an incomputable set~$A$ obeys~$\cc$, then~$\cc$ satisfies the limit condition.  \end{fact}

 \begin{proof} Suppose the limit condition fails for~$e$. There is  $s_0$ such 
that \bc $ \sum_{s\ge s_0} \sum_{x< s}  \cc_s(A_s)  \le \tp{-e}$. \ec To compute~$A$, on input~$n$ compute  $s >  \max (s_0, n)$ such that $ \cc(n,s) >  \tp{-e}$. Then $A_s(n) = A(n)$.  \end{proof}

\begin{convention}  \label{rmk:all finite} {\rm   For a \emph{monotonic} cost function $\cc$, we may  forthwith assume that $\ul \cc (x)< \infty$  for each $x$. For, firstly,    if  $\fao x [  \ul \cc (x) = \infty]$, then  $A \models \cc$ implies that $A$ is computable.  Thus,   we may   assume there is $x_0$ such that $ \ul \cc (x)$   is  finite for   all $x\ge x_0$ since $ \ul \cc (x)$ is nonincreasing. Secondly,  changing values $ \cc(x,s)$ for  the  finitely many $x< x_0$ does not alter the class of sets $A$ obeying~$\cc$. So  fix   some rational  $q>  \cc(x_0)$ and,   for $x< x_0$  redefine  $ \cc(x,s) = q$ for all~$s$.}  \end{convention}

	\subsection{The cost function for $K$-triviality}  \label{ss:standard_costfunction}   \ 
	
	\n  Let $K_s(x) = \min \{ \sssl \colon \, \UM_s(\sss) = x\}$ be the value of prefix-free descriptive string complexity of $x$ at stage~$s$.  We use  the conventions $K_s(x) = \infty $ for $x \ge s$  and $\tp{-\infty} =0$. Let  \begin{equation} \label{eqn:stcf} \cost(x,s) = \sum_{ w =x+1 }^s 2^{-K_s(w)}. \end{equation}  
	Sometimes $\cost$ is  called the \emph{standard cost function}, mainly  because it was  the first example of a cost function that received attention. Clearly,   $\cost$ is monotonic.
	Note that  
	 $\ul \cost(x) = \sum_{w>x} \tp{-K(w)}$. Hence  $\cost$ satisfies the limit condition: given $e\in \NN$, since $\sum _w \tp{-K(w)} \le  1$,
	there is an $x_0$ such that 
	 \bc  $\sum_{w\ge x_0}  \tp{-K(w)} \le  \tp{-e}$. \ec
	Therefore $\ulcost(x) \le  \tp{-e}$ for all $x \ge x_0$.

	The following example illustrates that in  Definition~\ref{df:obey cf},  obeying $\cost$, say,  strongly depends on the chosen enumeration.  Clearly, if we enumerate $A=\NN$ by putting in $x$ at stage $x$,  then the total cost of changes is zero. 
	\begin{prop}\label{ex:slow N} There is a   computable enumeration $\seq{A_s}_{s\in
	\NN}$ of $\NN$ in the order $0,1,2,\ldots$  (i.e., each $A_s$ is an initial segment of
	$\NN$) such that  $\seq{A_s}_{s\in
	\NN}$  does not obey~$\cost$.
	\end{prop}
	\begin{proof} Since $K(2^j) \lep 2 \log  j$, there is an increasing computable function~$f$ and a
	 number $j_0$ such that $\fa j \ge j_0 \, K_{f(j)}(2^j) \le j-1$.
	 Enumerate the set  $A=\NN$ in order, but so slowly that for each $j \ge j_0$ the elements of $(2^{j-1}, 2^j]$  are enumerated only after stage $f(j)$, one by one. Each such
	 enumeration costs at least $\tp{-(j-1)}$, so the  cost for each   interval $(2^{j-1}, 2^j]$ is 1. \end{proof} 

	Intuitively speaking,   an infinite c.e.\ set  $A$  can obey the cost function $\cost$ only  because during an   enumeration of $x$ at  stage $s$ one merely pays the current cost $\cost(x,s)$, not the limit  cost $\ulcost (x)$.

	 \begin{fact}  \label{fa: enumlate} If a c.e.\ set $A$  is infinite, then $\sum_{x \in A} \ulcost(x) = \infty$. \end{fact}

	 \begin{proof} Let $f$ be a 1-1 computable function with range $A$.  Let $L$ be the bounded request set $\{\la r, \max_{i \le \tp{r+1}} f(i)\ra \colon \, r \in \NN \}$. Let $M$ be a machine for $L$ according to the Machine Existence Theorem, also known as  the Kraft-Chaitin Theorem. See e.g.\ \cite[Ch.\ 2]{Nies:book} for background.  \end{proof} 

	In \cite{Nies:AM} (also see \cite[Ch.\ 5]{Nies:book}) it is shown that $A$ is $K$-trivial iff $A\models \cost$.   So far,  the class of $K$-trivial sets has been the  only known natural  class that is characterized by a single cost function.  However,   recent work with Greenberg and Miller suggests that for a c.e.\ set $A$, being below both halves $Z_0, Z_1$ of some  \ML-random  $Z = Z_0 \oplus Z_1$ is equivalent to obeying the cost function $\cc(x,s) = \sqrt{\Om_s - \Om_x}$.

	\subsection{Basic properties of the class of sets obeying a cost function}

	In this subsection, unless otherwise stated,    cost functions  will be monotonic. 
	Recall from Definition~\ref{def:proper cf}  that a cost function~$\cc$ is called \emph{proper} if  $\fao x \exo t  \cc(x,t) > 0$. 
	We investigate the class of models of a  proper  cost function~$\cc$. We also assume Convention~\ref{rmk:all finite} that $\ul \cc (x)< \infty$  for each $x$.

	The first  two results together  show that      $A\models \cc$  implies that  $A$ is weak truth-table below a c.e.\ set $C$ such that $C \models \cc$. 
Recall that a $\DII$ set $A$ is called $\omega$-c.e.\ if there is a computable approximation  $\seq{A_s}$ such that the number of changes $\# \{s\colon \, A_s(x) \neq A_{s-1}(x)\}$ is computably bounded in $x$; equivalently, $A \lwtt \Halt$ (see \cite[1.4.3]{Nies:book}).
	\begin{fact} Suppose that $\cc$ is a proper  monotonic cost function.  Let   $A \models \cc$. Then $A$ is $\omega$-c.e. \end{fact}

	\begin{proof} Suppose $\seq{A_s} \models  \cc$. Let $g$ be the computable function given by $g(x) = \mu t. \,  \cc(x,t) > 0$.   Let $\hat A_s(x) = A_{g(x)} (x)$ for $s< g(x)$, and $\hat A_s(x) = A_s(x) $ otherwise. Then the number of times $\hat A_s(x)$ can change is bounded by $\cc  {\seq{A_s}}  /  \cc(x,g(x))$. \end{proof}

	  Let $V_e$ denote the $e$-th $\omega$-c.e.\ set (see \cite[pg.\ 20]{Nies:book}). 
	\begin{fact} \label{ex:wtt index cost}  For each     cost function $\cc$, the index set  $\{e\colon \, V_e \models \cc\}$ is $\SI{3}$. 
	 \end{fact}
	 \begin{proof}  Let $D_n$ denote the $n$-th finite set of numbers. We may view the $i$-th partial computable function $\Phi_i$ as a (possibly partial) computable approximation   $\seq {A_t}$ by letting $A_t \simeq D_{\Phi_i(t)}$ (the symbol $\simeq$ indicates that 'undefined' is a possible value).  Saying   that $\Phi_i$ is total and a computable approximation of $V_e$ is a ~$\PI{2}$ condition of~$i$ and~$e$. Given that~$\Phi_i$ is total,   the condition that $\seq{A_t}\models \cc$  is~$\SI{2}$. 
	\end{proof}

	 The   {\it change set} (see \cite[1.4.2]{Nies:book}) for    a {computable approximation} $\seq{A_s}_{s\in \NN}$ of a  $\DII$ set~$A$ is a c.e.\ set  $C\ge_T A$   defined as follows:  if $s>0$  and  $A_{s-1}(x)\neq A_{s}(x)$ we put $\la x, i \ra$ into $C_{s}$, where $i $ is least such that $\la x, i \ra \not \in  C_{s-1}$.  If~$A$ is $\omega$-c.e.\ via this  approximation   then $C \ge_{tt} A$.  The change set can be used to prove the implication of the  Shoenfield   Limit Lemma that $A \in \DII$ implies $A\leT \Halt$; moreover,  if  $A$ is $\omega$-c.e.,   then $A \lwtt \Halt$. 

	\begin{prop}[\cite{Nies:book}, Section 5.3] \label{prop:c.e. ub cf}  Let the   cost function $\cc$ be non-increasing in the first component.  % such that  $  \cc(x,s)\ge  \cc(x+1,s)$ for each $x,s$.
	 If   a computable approximation $\seq{A_s}_{s\in \NN}$ of a set~$A$ obeys~$\cc$, then its  change set $C$   obeys~$\cc$ as well. \end{prop}

	\begin{proof}   Since $x < \la x,i\ra $ for each $x,i$, we have 
	\bc $C_{s-1}(x)
	\neq C_s(x) \ria     A_{s-1}\uhr x \neq  A_s \uhr x$ \ec for each $x,s$.
	 Then, since $ \cc(x,s)$ is non-increasing in~$x$, we have $\cc \seq{C_s}  \le \cc  {\seq{A_s}}  < \infty$. \end{proof}

This yields a limitation on the expressiveness of cost functions. Recall that $A$ is superlow if $A' \ltt \Halt$.
	\begin{cor}  \label{ex:superlow cost} There is no cost function $\cc$ monotonic in the first component  such that $A \models \cc$ iff $A$ is superlow.   \end{cor}

	\begin{proof} Otherwise,  for  each superlow set~$A$ there is a c.e.\ superlow set $C\geT A$.
	  This is clearly not the case: for instance~$A$ could be ML-random, and hence of diagonally non-computable degree, so that any c.e.\ set $C \ge_T A$ is Turing complete. \end{proof}

	 For $X \sub \NN$ let  $2X$ denote $\{2x\colon \, x \in X\}$. Recall that $A \oplus B = 2A \cup (2B +1)$.  
	 We now show that the class of sets obeying $\cc$ is   closed under $\oplus$ and  closed downward under a restricted form of   weak truth-table  reducibility.

	Clearly, $E \models \cc \lland F \models \cc$ implies $E \cup F \models \cc$. 

	  \begin{prop} Let the   cost function $\cc$ be monotonic in the first component. Then  $A \models \cc \lland B \models \cc$ implies $A \oplus B \models \cc$. \end{prop}
	 \begin{proof} Let $\seq{A_s} $ by a computable appoximation of $A$. 
	 By the monotonicity of $\cc$ we have $ \cc {\seq{A_s}}  \ge \cc {(2A_s)} $. Hence $2A\models \cc$. Similarly, $2B +1 \models \cc$. Thus $A\oplus B \models \cc$. \end{proof}

	 Recall that there are superlow c.e.\ sets $A_0,A_1$ such that $A_0 \oplus A_1$ is Turing complete (see \cite[6.1.4]{Nies:book}).  Thus the foregoing result yields a   a stronger form of Cor.\ \ref{ex:superlow cost}:  no cost function characterizes   superlowness within the c.e.\ sets. 
	 
	 \section{Look-ahead arguments} \label{s:look_ahead}
	
This core section of the   paper introduces an important   type of argument. Suppose we want to   construct a computable  approximation of a set $A$ that obeys  a given monotonic cost function. If we can anticipate  that $A(x)$ needs to be changed in the future, we try to   change   it  as early as possible, because earlier changes are   cheaper.  Such an argument  will be called  a  \emph{look-ahead argument}.   (Also see the remark  before  Fact~\ref{fa: enumlate}.) The main application of this method is to characterize logical properties of cost functions algebraically. 
	
\subsection{Downward closure under  $\le_{ibT}$}
Recall that $A \le_{ibT} B$ if $A\lwtt B$ with use function bounded by the identity. We now show that the class of models of $\cc$ is downward closed under $\le_{ibT}$.

	     \begin{prop} \label{prop:downward_closure} Let $\cc$ be a monotonic cost function. Suppose that $B\models \cc$ and $A = \Gamma^B$   via a Turing reduction $\Gamma$ such that  each   oracle query on an  input $x$ is at most $x$.    Then $A \models \cc$. \end{prop}

	 \begin{proof} 

 Suppose $B \models \cc$ via a computable approximation $\seq {B_s} \sN s$.  We define a computable increasing sequence of stages $\seq {s(i)} \sN i$ by  $s(0) = 0$
	 and \bc  $s(i+1) = \mu s > s(i) \, [ \Gamma^B \uhr {s(i)}[s]  \DA ]$. \ec
	 In other words, $s(i+1)$ is the least stage $s$ greater than $s(i)$  such that at stage~$s$, $\Gamma^B(n)$ is defined for each $n < s(i)$. 
	 We will  define $A_{s(k)}(x) $ for each $k\in \NN$. Thereafter  we let $A_s(x) = A_{s(k)}(x)$ where $k$ is maximal such that $s(k) \le s$.

	 Suppose $s(i) \le x < s(i+1)$.  For $k<i$ let  $A_{s(k)}(x)=v$,   where $v =\Gamma^B(x)[s(i+2)]$. For $k \ge i$, let $A_{s(k)}(x)=  \Gamma^B(x)[s(k+2)]$. (Note that these values  are  defined. Taking the $\Gamma^B(x)$ value at the large  stage $s(k+2)$ represents the look-ahead.)  									

	Clearly $\lim_s A_s(x) = A(x)$. We  show that  $\cc {\seq{A_s}}  \le \cc  {\seq{B_t}}$.
	 Suppose that $x$ is least such that $A_{s(k)}(x) \neq A_{s(k)-1}(x) $.   By the use bound on the reduction procedure $\Gamma$, there is   $y \le x$ such that $ B_t(y ) \neq B_{t-1}(y)$ for some~$t$, $ s(k+1)<  t \le  s(k+2)$.  Then  $ \cc(x, s(k))  \le  \cc(y,t) $ by monotonicity of~$\cc$.  Therefore $\seq { A_s} \models \cc$. \end{proof}

\subsection{Conjunction of  cost functions} In the remainder of this section we characterize conjunction and implication of monotonic cost functions algebraically. Firstly, 
we show that a set  $A$ is a model of $\cc$ and $\dd$ if and only if  $A$ is a model of $\cc + \dd$. Then we show that $\cc$ implies $\dd$ if and only if  $\ul \dd = O(\ul \cc)$.
	 %%%%%%%%%%%%%%%%%%%%%%%%%%%%%%%%%%%%%
	 \begin{thm}  \label{thm:conj}  Let $\cc,\dd $ be monotonic cost functions.  Then  
	 \bc $A \models \cc \lland A \models \dd \LLR A \models \cc+\dd$. \ec  \end{thm}

	 \begin{proof} \lapf This implication is trivial.

	\n  \rapf   We carry out   a   look-ahead argument of the type introduced in the proof of   Proposition~\ref{prop:downward_closure}.  Suppose that  $\seq {E_s}\sN s$ and $\seq {F_s} \sN s$ are computable approximations of  a set $A$ such that $\seq {E_s}  \models \cc$ and $\seq {F_s}  \models \dd$. We may assume that $E_s(x) = F_s(x) =0$ for $s<x$ because changing $E(x)$, say, to $1$ at stage $x$ will  not   increase the   cost   as $\cc(x,s) = 0$ for $x>s$.   We define a computable increasing sequence of stages $\seq {s(i)} \sN i$ by  letting  $s(0) = 0$
	 and \bc  $s(i+1) = \mu s > s(i) \, [ E_s \uhr {s(i)} = F_s \uhr {s(i)} ]$. \ec
	 We define $A_{s(k)}(x) $ for each $k\in \NN$. Thereafter we let $A_s(x) = A_{s(k)}(x)$ where $k$ is maximal such that $s(k) \le s$.

	 Suppose $s(i) \le x < s(i+1)$.  Let  $A_{s(k)}(x)=0$ for $k<i$.   To define $A_{s(k)}(x)$ for $k \ge i$, let $j(x)$ be the least $j \ge i$ such that  $v= E_{s(j+1)}(x)= F_{s(j+1)}(x)$.  

	\bc  $A_{s(k)}(x) =  \begin{cases} v & \text{if} \ i \le k \le j(x) \\
	 									E_{s(k+1)}(x) = F_{s(k+1)}(x) & \text{if} \, k > j(x). \end{cases}$ \ec

	Clearly $\lim_s A_s(x) = A(x)$. To show $(\cc+\dd)  \seq{A_s}  < \infty$,  suppose that $A_{s(k)}(x) \neq A_{s(k)-1}(x) $.  The only  possible cost in  the   case $ i \le k \le j(x) $ 
	is at stage $s(i)$ when $v=1$. Such a cost is bounded by $\tp{-x}$. XX Now consider a cost   in  the   case  $k > j(x)$. There is a least  $y $ such that $ E_t(y ) \neq E_{t-1}(y)$ for some $t$, $ s(k)<  t \le  s(k+1)$. Then $y \le x$, whence $ \cc(x, s(k))  \le  \cc(y,t) $ by the monotonicity of $\cc$. Similarly, using $\seq{F_s}$ one can  bound   the cost of changes due to $\dd$. Therefore 
 $   (\cc+\dd){\seq{A_s}} \le 4+ \cc   \seq {E_s}   +    \dd {\seq {F_s}} < \infty$.
	 \end{proof}
%
%  Theorem~\ref{thm:conj} also has   an infinitary version. Suppose  $\seq{\cc_i}\sN i$  is a uniformly computable sequence of monotonic cost functions with $\cc_i(0) \le 1$. Then $A$ obeys each $\cc_i$ if and only if  $A $ obeys  the monotonic cost function $\sum_i \tp{-i} \cc_i$.  

%The monotonicity in the stage is important. It is not hard 

 \subsection{Implication  between  cost functions}

\begin{definition} \label{def:cf_implication} {\rm      For cost functions  $\cc$ and $\dd $,   we  write $\cc \ria \dd$ if $A \models \cc$ implies $A \models \dd$ for each  ($\DII$) set $A$.} \end{definition}

  If a  cost function $\mathbf \cc$  is monotonic in the stage component,  then  $\ul \cc(x) = \sup_s  \cc(x,s)$. By Remark~\ref{rmk:all finite}   we  may  assume   $\ul \cc(x)$ is finite for each $x$.
      We will  show $\cc \ria \dd$  is equivalent to $ \ul \dd(x)  = O( \ul \cc(x))$.  In particular, whether or not $A \models \cc$ only depends on the limit function $ \ul \cc$.

  \begin{thm} \label{thm:imply} Let $\cc,\dd$ be  cost functions that are  monotonic in the stage component. Suppose $\cc$ satisfies the limit condition in  Definition~\ref{df:lc}.  Then 

\bc  $\cc \ria  \dd  \LLR  \exo N \fao x \big [ N  \cc(x) >   \dd(x) \big ]$. \ec  \end{thm} 
  \begin{proof} \lapf % Fix $N \in \NN$ such that $ \dd(x) < N  \cc(x)$ for each $x$. 
  We carry out yet another   look-ahead argument. We define a computable increasing sequence of stages $s(0) < s(1) < \ldots$ by  $s(0) = 0$
 and \bc  $s(i+1) = \mu s > s(i) . \fao x < s(i) \, \big [N  \cc(x,s) >   \dd(x,s) \big ]$. \ec
 Suppose $A$ is a $\DII$ set with a computable approximation  $\seq{A_s} \models \cc$.  We show that $\seq {\wt A_t} \models \dd$ for some computable approximation $\seq { \wt A_t}$ of $A$. As usual, we define $\wt A_{s(k)}(x) $ for each $k\in \NN$. We then  let $\wt A_s(x) = \wt A_{s(k)}(x)$ where $k$ is maximal such that $s(k) \le s$.
  
  Suppose $s(i) \le x < s(i+1)$.  If  $k <  i+1$ let   $\wt A_{s(k)}(x)=A_{s(i+2)}(x)$. If $k\ge  i+1$ let    $\wt A_{s(k)}(x)=A_{s(k+1)}(x)$.

 Given $k$,  suppose   that $x$ is least such that $\wt A_{s(k)}(x) \neq \wt A_{s(k)-1}(x)$. Let $i$ be the number such that $s(i) \le x  < s(i+1)$. Then $k \ge i+1$. We have $A_t(x) \neq A_{t-1}(x)$ for some $t$ such that $s(k) < t \le s(k+1)$. Since $x< s(i+1) \le s(k)$, by the monotonicity hypothesis this implies $N \cc(x,t) \ge N  \cc(x,s(k) )>  \dd(x, s(k))$.  So  $\dd  {\seq{\wt A_s}}   \le  N \cdot  \cc  {\seq{A_s}}  < \infty$.   Hence $A \models \dd$.
  
  \vsp

\n \rapf Recall from the proof of Fact~\ref{ex:wtt index cost} that we   view the $e$-th  partial computable function $\Phi_e$ as a (possibly partial) computable approximation   $\seq{B_t}$, where  $B_t \simeq D_{\Phi_e(t)}$.  

Suppose that $ \exo N \fao x  \, \big [   N \ul  \cc(x) >  \ul  \dd(x) \big ]$ fails. We build  a set $A \models \cc$ such that for no  computable approximation $\Phi_e$ of $A$  we have $\dd \, \Phi_e  \le 1$. This suffices for the theorem by  Fact~\ref{easy1}.
  We meet the requirements 

\bc $R_e\colon \, \Phi_e \ttext{is total and approximates} A  \RRA  \Phi_e \not \models \dd.$ \ec
  
  The idea is to change $A(x)$   for some fixed $x$ at sufficiently many stages $s$ with $N  \cc(x,s) <  \dd(x,s)$, where $N$ is   an appropriate   large constant. After each change we wait for recovery from the side of $\Phi_e$. In this way our $\cc$-cost of changes to  $A$  remains bounded,  while  the opponent's  $\dd$-cost of changes to $\Phi_e$ exceeds~$1$.

For a stage $s$, we let    $ \init_s(e) \le s$ be the  largest stage such that $R_e$ has been initialized   at that stage (or $0$ if there is no such stage).   Waiting for recovery is implemented as follows. We say that $s$ is \emph{$e$-expansionary} if $s=\init_s(e)$, or $s> \init_s(e) $ and,   where $u$ is the greatest $e$-expansionary stage less than $s$, 
\bc  $\ex t \in [u,s) \, [ \Phi_{e,s}(t) \DA \lland \Phi_{e,s}(t) \uhr u = A_u \uhr u]$. \ec
  The strategy for $R_e$  can only change $A(x)$  at  an   $e$-expansionary stage $u$ such that  $x< u$. In this case it preserves $A_u \uhr u$ until  the next $e$-expansionary stage. Then,  $\Phi_e $ also  has to    change its mind on~$x$: we have \bc  $x\in \Phi_e(u-1) \leftrightarrow  x \not \in \Phi_e(t)$ for some $t \in [u,s)$.  \ec

  We measure the progress of $R_e$ at stages $s$ via a quantity $ \aaa_s(e)$. When $R_e$ is initialized at stage $s$, we set $ \aaa_s(e)$ to~$0$. 
If  $R_e$   changes $A(x)$    at stage $s$, we increase $ \aaa_s(e)$ by $\cc(x,s)$. $R_e$ is declared satisfied when   $ \aaa_s(e)$ exceeds  $\tp{-b-e}$, where $b$ is the number of times $R_e$ has been initialized.

\vsp

\n {\it Construction of $\seq{A_s}$ and $\seq { \aaa_s}$.}
Let $A_0= \ES$. Let $\aaa_{0}(e)=0$ for each $e$. 

\n {\it Stage $s> 0$.} Let $e$ be least such that $s$ is $e$-expansionary and  $\aaa_{ s-1}(e) \le  \tp{-b-e}$ where $b$ is the number of times $R_e$ has been initialized so far. If $e$ exists do the following. 

Let $x$ be least such that  $\init_s(e) \le x < s $, $ \cc(x,s) < \tp{-b-e}$  and 
\bc $\tp{b+e}  \cc(x,s) <  \dd(x,s)$. \ec
If $x$ exists let $A_s(x)= 1-A_{s-1}(x)$. Also let $A_s(y) = 0$ for $x< y < s$.  Let $\aaa_s(e)= \aaa_{s-1}(e)+  \cc(x,s)$. Initialize the requirements $R_i$ for $i>e$ and let  $\aaa_s(i)=0$. (This preserves $A_s\uhr s$ unless $R_e$ itself is later initialized.) We say that $R_e$ \emph{acts}. 

\vsps

\n \verif  If $s$ is a stage such that $R_e$ has been initialized for~$b$ times,  then  $\aaa_s(e) \le \tp{-b-e+1}$. Hence the total cost of changes of $A$ due to  $R_e$ is at most  $\sum_b \tp{-b-e+1}= \tp{-e+2}$. Therefore $\seq{A_s} \models \cc$. 

We show that {\it each $R_e$ only acts finitely often,  and is met.} Inductively, $\init_s(e)$ assumes  a final value $s_0$. Let $b$ be the number of times $R_e$ has been initialized by stage  $s_0$.  

Since  the  condition $ \exo N \fao x [ N  \ul \cc(x) >  \ul \dd(x)]$ fails, there is  $x \ge s_0$ such that   for some $s_1 \ge x$, we  have $\fa s \ge s_1 \, [\tp{b+e}  \cc(x,s) <  \dd(x,s)]$. Furthermore, since $\cc$ satisfies the limit condition, we may suppose that   $ \ul \cc(x) < \tp{-b-e}$.  Choose $x$ least. 

If~$\Phi_e$ is a computable approximation of $A$,    there are infinitely many $e$-expansionary stages $s \ge s_1$. For each such $s$,  we can choose this     $ x$ at stage $s$ in the construction. So we can add at least $ \cc(x,s_1)$ to $\aaa(e)$. Therefore   $\aaa_t(e)$ exceeds the bound $\tp{-b-e}$ for   some stage $t\ge s_1$, whence   $R_e$ stops acting at $t$. Furthermore, since $\dd$ is monotonic  in the second component and by the initialization due to $R_e$,  between  stages $s_0$ and  $t$ we have caused  $\dd \, \Phi_e $ to increase by  at least $\tp{b+e} \aaa_t(e) >  1$. Hence $R_e$ is met. 
  \end{proof} 
  
  The foregoing proof uses in an essential way the ability to change $A(x)$, for the same $x$, for a multiple number of  times.   If we restrict implication to  c.e.\ sets, the   implication from left to right in Theorem~\ref{thm:imply} fails. For a trivial example,   let $ \cc(x,s)= 4^{-x}$ and $ \dd(x,s) = \tp{-x}$. Then each c.e.\ set obeys $\dd$, so $\cc \to \dd$ for c.e.\ sets. However, we do not have $ \dd(x)= O( \cc(x))$. 

We mention  that Melnikov and Nies (unpublished, 2010) have obtained a sufficient  algebraic condition  for the  non-implication  of cost functions via a  c.e.\ set. Informally speaking, the condition $ \dd(x)= O( \cc(x))$ fails ``badly''.
  \begin{proposition}
	Let $\cc$ and $\dd$  be monotonic cost functions satisfying the limit condition such that
  $\sum_{x \in \NN}  \ul \dd(x) = \infty$ and, for each $N>0$, 
\[ \sum  \ul \dd (x) \Cyl{ N\ul \cc (x)> \ul \dd (x)} < \infty. \]
	Then there exists a c.e. set $A$   that   obeys $\cc$, but not  $\dd$. 
	\end{proposition}
	 The  hope is  that some variant of this will yield  an algebraic criterion for   cost function implication   restricted   to the  c.e.\ sets.

\section{Additive cost functions}  \label{s:additive cf}
 We discuss  a class of very simple cost functions introduced  in \cite{Nies:ICM}. We  show that a  $\DII$ set obeys all of them if and only if it is~$K$-trivial.   There   is a universal cost function of this kind, namely $ \cc(x,s) = \Om_s - \Om_x$.   Recall Convention~\ref{rmk:all finite} that $\ul \cc(x)< \infty$ for each cost function $\cc$. 
 \begin{deff}[\cite{Nies:ICM}] We say that a cost function $\cc$ is \emph{additive} if  $\cc  (x,s) =0$ for $x >s $, and   for each $x< y<z$ we have 
 \bc  $ \cc(x,y) +  \cc(y,z)=  \cc(x,z)$. \ec \end{deff}
 Additive cost functions correspond  to nondecreasing effective sequences $\seq {\beta_s}\sN s$ of non-negative rationals, that is, to effective approximations of left-c.e.\ reals $\beta$. Given  such an approximation  $\langle \beta \rangle = \seq {\beta_s}\sN s$, let for $x \le s$
 \bc $\cc_{\langle \beta \rangle} (x,s) = \beta_s - \beta_x$. \ec 
 Conversely, given an additive cost function $\cc$, let $\beta_s=  \cc(0,s)$. Clearly the two effective  transformations are   inverses of each other.

\subsection{$K$-triviality and the   cost function $\cc_{\seq \Om}$}
The standard cost function $\cost$  introduced in (\ref{eqn:stcf})  is \emph{not} additive. We certainly  have $\cost(x,y) + \cost(y,z)\le  \cost(x,z)$, but by stage $z$ there could be a shorter description of, say, $x+1$ than at stage $y$, so that the inequality may be proper. On the other hand, let $g$ be a computable function such that $\sum_w \tp{-g(w)} < \infty$; this implies that $K(x) \lep g(x)$. The ``analog'' of $\cost $ when we write  $g(x) $ instead of $K_s(x)$, namely  $\cc_g(x,s)= \sum_{w=x+1}^s \tp{-g(w)}$ is an additive cost function. %, where $\beta_s =\sum_{w=0}^s \tp{-g(w)}$. 

 Also, $\cost$  is dominated by an additive cost function $\cc_{\langle \Om \rangle}$ we introduce next. Let $\UM$ be the standard universal prefix-free machine (see e.g.\ \cite[Ch.\ 2]{Nies:book}). Let  $\langle \Om \rangle$ denote   the computable approximation of $\Om$  given by $\Om_s  = \leb \, \dom (\UM_s)$. (That is, $\Om_s$ is the Lebesgue measure of the domain of the universal \PF   machine at stage~$s$.)
 
 \begin{fact} \label{fa:Omfact} For each $x\le  s$, we have $\cost(x,s) \le \cc_{\seq \Om}(x,s) = \Om_s - \Om_x$. \end{fact}
 
 \begin{proof} Fix $x$. We prove the statement  by induction on $s\ge x$. For $s=x$ we have  $\cost(x,s) = 0$. Now \bc  $\cost(x,s+1) - \cost(x,s) =  \sum_{ w =x+1 }^{s+1}  2^{-K_{s+1}(w)}-  \sum_{ w =x+1 }^s 2^{-K_s(w)} \le \Om_{s+1} - \Om_s$,  \ec because the difference    
 is due to convergence at stage $s$ of new  $\UM$-computations. 
 \end{proof} 
 \begin{thm} Let $A$ be $\DII$. Then the following are equivalent. 
 
 \bi \itone   $A$ is $K$-trivial.
 \ittwo   $A$ obeys each additive cost function. 
 \itthree $A$ obeys  $\cc_{\seq \Om}$, where  $\Om_s = \leb \dom ( \UM_s)$.  \ei \end{thm}
 
 \begin{proof} 
 (ii) $\ria$ (iii) is immediate, and (iii) $\ria$ (i) follows from Fact~\ref{fa:Omfact}. It remains to show (i)$\ria$(ii).

 Fix some computable approximation $\seq{A_s}\sN{s}$ of~$A$.  Let $\cc$ be an additive cost function. We may suppose that $  \ul \cc(0) \le 1$. 

For $w>0$ let $r_w \in \NN \cup {\infty} $ be least such that $\tp{-r_w} \le c(w-1,w)$ (where  $\tp{-\infty } =0$).  Then $\sum_w \tp{-r_w} \le 1$. Hence by the Machine Existence Theorem we have $K(w)\lep r_w$ for each $w$. This implies $\tp{-r_w} = O(\tp{-K(w)})$, so $\sum_{w> x} \tp{-r_w} = O(\ul \cost(x))$ and hence $\ul \cc(x) = \sum_{w>x} c(w-1,w) = O(\ul \cost(x))$. Thus $\cost \to \cc$ by Theorem~\ref{thm:imply}, whence the  $K$-trivial set $A$ obeys $\cc$.  (See \cite{Bienvenu.Greenberg.ea:16} for a proof not relying on Theorem~\ref{thm:imply}.)
  \end{proof} 
  
Because of  Theorem~\ref{thm:imply}, we have $\cc_{\seq \Om} \leftrightarrow \cost$. That is,  
  
  \bc $\Om - \Om_x \sim \sum_{w= x+1}^\infty \tp{-K(w)}$. \ec
This can easily  be seen directly: for instance, $\cost  \le \cc_{\seq \Om}$ by Fact~\ref{fa:Omfact}.

\subsection{Solovay reducibility}
  Let $\Qbinary $ denote the dyadic rationals, and let the variable $ q$ range over $\Qbinary$. Recall Solovay reducibility on left-c.e.\ reals: $\beta \le_S \aaa$ iff there is a partial computable $\phi \colon \,  \Qbinary \cap [0,\aaa) \ria \Qbinary \cap [0,\beta) $ and $N \in \NN$ such that  \bc $\fao q < \aaa \big [ \beta - \phi(q) < N(\aaa - q)\big ]$. \ec  Informally, it  is easier to approximate  $\beta$  from the left, than $\aaa$.   See e.g.\ \cite[3.2.8]{Nies:book} for background.

We will show that reverse   implication  of additive cost functions corresponds to Solovay reducibility on the corresponding left-c.e.\ reals. Given a left-c.e.\ real $\gamma$, we let the variable  $\seq \gamma$ range over the  nondecreasing effective sequences  of rationals converging to $\gamma$.

  \begin{prop}  \label{prop:heyheyey} Let $\aaa, \beta$ be left-c.e.\ reals. The following are equivalent.  

   \bi \itone   $\beta \le_S \aaa$

    \ittwo $  \fa \seq  \aaa \ex \seq \beta \, [c_{\seq \aaa} \ria c_{\seq \beta}]$ 
 \itthree $  \ex \seq  \aaa \ex \seq \beta \, [ c_{\seq \aaa} \ria c_{\seq \beta}] $.  \ei 
 \end{prop}

 \begin{proof} 

 \n (i) $\ria$ (ii). Given an effective sequence  $\seq \aaa$, by the definition of $\le_S$ there is  an  effective sequence~$\seq \beta$ such that $\beta - \beta_x = O(\aaa - \aaa_x)$ for each $x$. Thus $\ul \cc_{\seq \beta}   = O( \ul \cc_{\seq \aaa}  )$. Hence  $\cc_{\seq \aaa } \ria \cc_{\seq \beta}$ by  Theorem~\ref{thm:imply}.

 \vsps

   \n (iii) $\ria$ (i).  Suppose  we are given $\seq \aaa$ and $\seq \beta$ such that $\ul \cc_{\seq \beta}   = O( \ul \cc_{\seq \aaa}  )$. 
  Define a partial computable function $\phi$ by $\phi(q) = \beta_x$ if $\aaa_{x-1} \le q < \aaa_x$. Then $\beta \le_S \aaa$ via $\phi$. 
 \end{proof}

\subsection{The strength of an  additive cost function}

Firstly, we make some remarks related to Proposition~\ref{prop:heyheyey}. For instance, it implies that an  additive cost function can  be weaker than $\cc_{\la \Om \ra}$ without being  obeyed by all  the~$\DII$ sets.
  \begin{prop} There are additive cost functions $\cc,\dd$ such that $\cc_{\la \Om \ra} \ria \cc$, $\cc_{\la \Om \ra} \ria \dd$ and $\cc,\dd$ are incomparable under the implication of cost functions.  \end{prop}

  \begin{proof} Let $\cc,\dd$ be cost functions corresponding to enumerations of Turing (and hence Solovay) incomparable left-c.e.\ reals. Now apply Prop.~\ref{prop:heyheyey}.  \end{proof}

  Clearly, if $\beta$ is a  computable real then any c.e.\  set obeys $\cc_{\la \beta \ra}$.  The intuition we garner from Prop.\ \ref{prop:heyheyey} is that a more complex left-c.e.\ real $\beta$ means that the sets $A \models \cc_{\la \beta \ra}$ become less complex, and conversely. We  give  a little more    evidence for this principle: if $\beta$ is non-computable, we show that a set $A \models \cc_{\la \beta \ra}$ cannot be weak truth-table complete. However, we also build a  non-computable $\beta$ and  a c.e.\ Turing complete set that  obeys $\cc_{\la \beta \ra}$

\begin{prop}
	Suppose  $\beta$ is a non-computable left-c.e.\ real and     $A \models \cc_\seq \beta$. 
	Then  $A$ is not weak truth-table complete.  
\end{prop}

\begin{proof} Assume for a contradiction that $A$ is weak truth-table complete.  We  can fix a computable approximation $\seq  {A_s}$ of $A$ such that $ \cc_\seq \beta \seq  {A_s} \le 1$. We build a c.e.\ set $B$. By the recursion theorem we can suppose we have a weak truth-table reduction $\Gamma$ with computable use bound $g$ such that $B = \Gamma^A$.   We build $B$ so  that $\beta -  \beta_{g(\tp{e+1})} \le \tp{-e}$, which implies that $\beta$ is computable. 

Let $I_e = [2^e, 2^{e+1})$.  If ever  a stage $s$ appears such that $\beta_s -  \beta_{g(\tp{e+1})} \le \tp{-e}$, then we start  enumerating into   $B \cap I_e$ sufficiently slowly so that $A \uhr {g(2^{e+1})}$ must change $2^e$ times. To do so, each time we enumerate into  $B$, we wait for a recovery of $B = \Gamma^A$ up to  $2^{(e+1)}$.  The $A$-changes we enforce yield a total cost $>1$ for a contradiction.
\end{proof} 
\begin{prop} There is a non-computable left-c.e.\ real $\beta$ and a c.e.\ set  $A \models \cc_\seq \beta$
such that $A$  is Turing complete. 
\end{prop}
\begin{proof} We build a Turing reduction $\Gamma$ such that $\Halt = \Gamma(A)$. Let $\gamma_{k,s}+1$ be the use of the computation $\Gamma^\Halt(k)[s]$. We  view $\gamma_{k}$ as a movable marker as usual. The initial value is $\gamma_{k,0} = k$. Throughout the construction we maintain the invariant   
\bc $\beta_s - \beta_{\gamma_{k,s}} \le \tp{-k}$. \ec
Let $\seq {\phi_e}$ be the usual effective list of partial computable  functions. By convention, at each stage at most one computation $\phi_e(k)$ converges newly.
To make $\beta$ non-computable,  it suffices to  meet the requirements
\bc $R_k\colon \,   \phi_k(k) \DA \RRA   \beta - \beta_{\phi_k(k)} \ge \tp{-k}$. \ec
\n {\it Strategy for $R_k$.} If $\phi_k(k) $ converges newly at stage $s$, do the following.
\bi \item[1.] Enumerate $\gamma_{k,s}$ into $A$. (This incurs a cost of at most  $\tp{-k}$.)
\item[2.]  Let $\beta_s = \beta_{s-1} + \tp{-k}$.
\item[3.] Redefine $\gamma_i$ ($i\ge k$) to large values in an increasing fashion.
\ei
In the construction, we run the strategies for the $R_k$. If  $k$ enters $\Halt$ at stage $s$, we enumerate $\gamma_{k,s}$ into $A$.

Clearly each $R_k$ acts at most once, and is met. Therefore $\beta$ is non-computable. The markers $\gamma_k$ reach a limit. Therefore $\Halt = \Gamma(A)$. Finally, we maintain the stage invariant, which implies that the total cost of enumerating $A$ is at most $4$. 
\end{proof} 
%Further research is needed here. For instance, if $\beta$ is Turing complete and $A \models \cc_\seq \beta$, is $A$ low?
As pointed out by Turetsky, it can be verified  that $\beta$ is in fact  Turing complete. 

Next,  we note that if we have two computable  approximations from the left of the same real, we obtain  additive cost functions with  very similar classes of models.
   \begin{prop} Let $\seq {\aaa}, \seq \beta$ be left-c.e.\ approximations of the same real.  Suppose that $A \models \cc_{\seq \aaa}$. Then there is $B \equiv_m A$ such that $ B \models \cc_{\seq \beta}$. If $A$ is c.e.,  then $B$ can be chosen  c.e.\ as well. \end{prop}  
   
 \begin{proof} Firstly,   suppose  that $A$ is c.e. By Fact~\ref{fact:fukd} choose a computable enumeration $\seq{A_s} \models \cc_{\seq \aaa}$.  
 
 By the hypothesis on the sequences $\seq {\aaa}$ and $ \seq \beta$,   there is a computable  sequence of stages  $s_0 < s_1 < \ldots$   such that $|\aaa_{s_i} - \beta_{s_i}|  \le \tp{-i}$. Let $f$ be a strictly increasing computable function such that $\aaa_x \le \beta_{f(x)}$ for each~$x$.  
 
To define $B$,  if $x $ enters $A$ at stage $s$, let $i$ be greatest such that $s_i \le s$. If $f(x) \le s_i$   put $f(x)$ into $B$ at stage $s_i$.

 Clearly \bc $\aaa_s - \aaa_x \ge \aaa_{s_i} - \aaa_x \ge \aaa_{s_i} - \beta_{f(x)}  \ge \beta_{s_i} - \beta_{f(x) } - \tp{-i}$. \ec So $ \cc_{\seq \beta} \seq {B_s} \le  \cc_{\seq \aaa} \seq {A_s}+ \sum_i \tp{-i}$.

 Let $R$ be the computable subset of $A$ consisting of those $x$  that are enumerated early,  namely  $x $ enters $A$ at  a stage $s$ and  $ f(x)  > s_i$ where  $i$ is greatest such that $s_i \le s$. Clearly $B = f(A-R)$. Hence $B \equiv_m A$.

 The argument can be adapted to the case that $A$ is $\DII$. Given a computable approximation $\seq{A_s}$ obeying $\cc_{\seq \aaa }$, let $t$ be the least $s_i$  such that  $s_i \ge f(x)$.  For $s \le t$ let $B_s(f(x)) = A_t(x)$. For $s > t$ let   $B_s(f(x)) = A_{s_i}(x)$   where $s_i \le s < s_{i+1}$. \end{proof}

\section{Randomness, lowness, and $K$-triviality}
\label{s:random_low_K-triv}

Benign  cost functions  were briefly discussed in the introduction.
\begin{deff}[\cite{Greenberg.Nies:11}] \label{df:benign} {\rm  A monotonic cost function $\cc$ is called   {\it benign} if  there is a computable function~$g$ such that  for all $k$, 
   \bc    $x_0< x_1 < \ldots < x_k  \lland  \fa i < k \, [ \cc(x_i, x_{i+1}) \ge \tp{-n}]$ implies $k \le g(n)$. \ec } \end{deff} 
   %  
%   In other words, $g(n)$ bounds the size of any collection of pairwise disjoint intervals $[x,s)$ such that $ \cc(x,s) \ge \tp{-n}$.  
Clearly such a cost function  satisfies the limit condition. Indeed, $\cc$ satisfies the limit condition if and only if  the above holds for some $g \leT \Halt$. For example,  the cost function $\cost $ is benign via  $g(n) = 2^n$.
  Each additive cost function is benign  where $g(n) = O(2^n)$.  For more detail see  \cite{Greenberg.Nies:11} or~\cite[Section~8.5]{Nies:book}.

For definitions and background on the extreme lowness property called strong jump traceability, see \cite{Greenberg.Nies:11,Greenberg.Hirschfeldt.ea:12} or \cite[Ch.\ 8 ]{Nies:book}. 
	We will use the main result in  \cite{Greenberg.Nies:11} already quoted in the introduction: a c.e.\ set $A$ is strongly jump traceable iff $A$ obeys each benign cost function.

	\subsection{A cost function implying  strong jump traceability} 
	The following type of cost functions first appeared in \cite{Greenberg.Nies:11} and \cite[Section 5.3]{Nies:book}. Let $Z \in \DII$ be ML-random. Fix a computable approximation $\seq{Z_s}$ of $Z$ and let $\cc_Z$ (or, more accurately, $\cc_\seq {Z_s}$) be the cost function defined  as follows. Let 
	$\cc_Z(x,s) = \tp{-x} $ for each $x\ge s$; if $x<s$,  and~$e< x$ is
	 least such that
	$Z_{s-1} ( e) \neq Z_s  ( e)$, we     let 
	 \begin{equation} \label{eqn: cZ}  c_Z(x,s) =
	\max(c_Z(x,s-1), \tp{-e} ). \end{equation}  Then $A \models \cc_Z$ implies $A \leT Z$ by the aforementioned result
	from  \cite{Greenberg.Nies:11}, which is proved like its variant above. 

	A {Demuth test}  is a sequence of c.e.\ open sets $(S_m)\sN{m}$ such that  
\bi  \item $\fao m \leb S_m \le \tp{-m}$, and there is a function  $f$ such that  $S_m $ is the $\SI 1$  class   $\Opcl{W_{f(m)}}$; \item  $f(m) = \lim_s g(m,s)$ for a computable function $g$ such that the size of the set $\{s\colon \, g(m,s) \neq g(m,s-1)\}$ is bounded by a computable function $h(m)$.  \ei
 
  A set~$Z$ {passes} the test if $  Z\not \in   S_m$ for almost every~$m$.  
   We say that~$Z$ is  {Demuth random}  if~$Z$ passes  each Demuth test.  
  For background on Demuth randomness see \cite[pg.\ 141]{Nies:book}. 

	% \Kuc\  and Nies \cite{Kucera.Nies:nd} have shown that each c.e.\ set  $A$  Turing below a Demuth random  set $Y$  is strongly jump traceable. The following easier fact works for any  $\DII$ sets $A, Y$. 
	\begin{prop} Suppose $Y$ is a  Demuth random   $\DII$ set and  $A \models  c_Y$. Then $A \leT Z$ for each $\omega$-c.e.\ ML-random set $Z$.  \end{prop} 
		In particular, $A$ is strongly jump traceable by \cite{Greenberg.Hirschfeldt.ea:12}.
	\begin{proof} Let $G^s_e = [ Y_t \uhr e]$ where $t\le s $ is greatest such that $Z_t(e) \neq Z_{t-1}(e)$. Let $G_e = \lim_s G^s_e$. (Thus, we only update $G_e$ when $Z(e)$ changes.) Then $(G_e) \sN e$ is a Demuth test. Since $Y$ passes this test,   there is $e_0$ such that 
	\bc $\fa e \ge e_0 \,  \fao t  [ Z_t(e) \neq Z_{t-1}(e) \ria  \ex s > t \  Y_{s-1} \uhr e \neq Y_s \uhr e] $. \ec
	We use this fact to define a computable approximation $(\hat Z_u)$ of $Z$  as follows:  let $\hat Z_u(e) = Z(e)$ for $e \le  e_0$; for $e > e_0$ let $\hat Z_u(e) = Z_s(e)$ where $s\le u$ is greatest such that  $Y_{s-1} \uhr e \neq Y_s \uhr e $. 

	Note that  $\cc_{\hat Z}(x,s) \le c_Y(x,s)$ for all $x, s$. Hence $A \models \cc_{\hat Z}$ and therefore $A \leT Z$. \end{proof}

	Recall that some  Demuth random set is  $\DII$. \Kuc\  and Nies \cite{Kucera.Nies:11} in their main result  strengthened  the foregoing proposition    in the case of a  c.e.\ sets $A$:   if  $A \leT Y$ for some Demuth random set $Y$, then  $A$ is strongly jump traceable.   Greenberg and Turetsky \cite{Greenberg.Turetsky:14} obtained the converse of this result: every c.e.\ strongly jump traceable is below a Demuth random. 

	\begin{remark} \label{spoon} {\rm  For each $\DII$ set $Y$ we have  $\cc_Y(x)= \tp{-F(x)}$ where $F$ is the $\DII$ function such that  \bc $F(x) = \min \{e\colon \, \ex s >x \, Y_s(e) \neq Y_{s-1}(e)\}$. \ec Thus $F$ can be viewed as a  modulus function in the sense of~\cite{Soare:87}. } 
	\end{remark}

	For a computable approximation $\Phi$ define   the cost function $\cc_\Phi $ as in (\ref{eqn: cZ}).  The following (together with Rmk.\ \ref{spoon}) implies that any computable approximation $\Phi$  of a  ML-random Turing incomplete      set    changes    late at small numbers, because the convergence of $\Om_s$ to $\Om$ is slow. %The view \cite{}
	\begin{cor}
	Let $Y <_T \Halt $ be a  ML-random set. Let $\Phi$ be any  computable approximation of $Y$. Then $\cc_{\Phi} \ria \cost$ and therefore $O(c_\Phi(x)) = \cc_{\seq \Om}(x)$. \end{cor}
	\begin{proof} If $A \models \cc_\Phi$ then $C \models \cc_\Phi$ where $C\geT A$ is the change set of the given approximation of $A$ as in Prop.\ \ref{prop:c.e. ub cf}. By \cite{Hirschfeldt.Nies.ea:07} (also see \cite[5.1.23]{Nies:book}),  $C$ and therefore $A$ is $K$-trivial. Hence $A \models \cc_{\seq \Om}$. \end{proof}

\subsection{Strongly jump traceable sets and d.n.c.\ functions} 
  Recall that  we write $X \le_{ibT} Y$ if $X \leT  Y$ with use function bounded by the identity.  When building prefix-free machines, we use the terminology of  \cite[Section 2.3]{Nies:book}   such as Machine Existence Theorem (also called the Kraft-Chaitin Theorem), bounded request set etc.

\begin{thm} \label{thm:Kuc-g} Suppose an  $\omega$-c.e.\ set $Y$  is diagonally noncomputable via a function that is weak truth-table below $ Y$. Let $A$ be a strongly jump traceable c.e.\ set.  Then $A \le_{ibT} Y$. \end{thm}

 \begin{proof} By \cite{Kjos.ea:2005} (also see \cite[4.1.10]{Nies:book}) there is an order function $h$ such that $2h(n) \lep K(Y\uhr n)$ for each $n$. 
The argument of the present  proof goes back to \Kuc's injury free solution to Post's problem (see \cite[Section 4.2]{Nies:book}). Our proof is phrased in  the language of cost functions,   extending  the similar result in \cite{Greenberg.Nies:11} where  $Y$ is ML-random   (equivalently, the condition above holds with $h(n)= \lfloor n/2 \rfloor +1$.   

 Let $\seq {Y_s}$ be a computable approximation via which $Y$ is $\omega$-c.e. To help with  building a reduction procedure for   $A \le_{ibT} Y$, via the Machine Existence Theorem  we give  prefix-free  descriptions of initial segments $Y_s\uhr e$. On input $x$,  if   at a   stage  $s>x$, $e$ is least such that  $Y(e)$ has changed  between stages $x$ and $s$, then we  still  hope  that    $Y_s \uhr e$  is   the final version of $Y \uhr e$.  So whenever $A(x)$ changes at such a stage $s$, we  give a description of  $Y_s\uhr e$   of length $h(e)$.  By hypothesis   $A$ is strongly jump traceable,  and hence obeys each  benign cost function.   We define an appropriate benign cost function  $\cc$  so that a set  $A $  that obeys~$\cc$     changes little enough that   we can provide all the descriptions  needed.

  To ensure that   $A \le_{ibT} Y$,  we define a computation  $ \Gamma(Y \uhr x) $  with output $A(x)$ at  the least  stage  $t\ge x $ such that $Y_t\uhr x$ has the final value.  If  $Y$ satisfies the hypotheses of the theorem, $A(x) $ cannot change at any stage  $s> t$ (for almost all $x$), for otherwise $Y \uhr e$ would receive  a description of length $h(e)+O(1)$,  where $e$ is least such that 
  $Y(e)$ has changed  between  $x$ and  $s$.
   
We give the details. Firstly  we give a definition of  a   cost
 function $\cc$ which generalizes the definition in (\ref{eqn: cZ}).   Let   $ \cc(x,s) = 0 $ for each $x\ge s$.  If  $x<s$,  and~$e< x$ is
 least such that
$Y_{s-1} ( e) \neq Y_s  ( e)$,     let 
 \begin{equation} \label{eqn:defn of cY}  \cc(x,s) =
\max( \cc(x,s-1), \tp{-h(e)} ). \end{equation} 
Since $Y$ is $\omega$-c.e.,  $\cc$ is benign. Thus each strongly jump traceable  c.e.\ set obeys $\cc$ by the main result in  \cite{Greenberg.Nies:11}. So  it suffices to show that $A \models \cc$ implies $A \le_{ibT} Y$ for any set $A$. Suppose that $\cc{\seq{A_s}}  \le 2^u$. Enumerate a bounded request set~$L$ as follows. When $A_{s-1}(x)\neq A_s(x)$ and $e$ is least such that $e=x$ or  $Y_{t-1} ( e) \neq Y_t  ( e)$ for some $t \in [x,s)$, put the  request $\la u+ h(e), Y_s \uhr e \ra$ into $L$. Then $L$ is indeed a bounded request set.

Let $d$ be a  coding constant for $L$ (see \cite[Section 2.3]{Nies:book}). Choose $e_0$ such that $h(e) + u+ d  < 2h(e)$ for each $e \ge e_0$.  Choose $s_0\ge e_0$ such that $Y \uhr{e_0}$ is stable from stage $s_0$ on.

To show $A \le_{ibT} Y$, given an input $x\ge s_0$, using~$Y$ as an oracle,  compute $t >x$ such that $Y_t \uhr x =Y \uhr x$.
We claim that  $A(x) = A_t(x)$. Otherwise  $A_{s}(x) \neq A_{s-1}(x)$ for some $s > t$. Let $e \le x$ be the  largest number such that $Y_r\uhr e = Y _t \uhr e$ for all~$r$ with  $t < r \le s$.  If $e <  x$  then   $Y(e)$ changes in the  interval $(t,s]$ of stages.  Hence, by the choice of $t\ge s_0$, we cause $K(y)  < 2h(e)$ where  $y= Y_t \uhr {e} = Y\uhr {e}$, contradiction.   \end{proof}
%%%%%%
\begin{example} {\rm  For each order function $h$ and constant $d$,  the class
\bc $P_{h, d }= \{Y \colon \, \fao n 2h(n) \le K(Y\uhr n)+d \}$ \ec
is $\PPI$. Thus,  by the foregoing proof, each strongly jump traceable c.e.\ set is \emph{ibT} below each $\omega$-c.e.\ member of $P_{h,d}$. } \end{example}

We discuss  the foregoing  Theorem~\ref{thm:Kuc-g}, and relate it   to results in \cite{Greenberg.Hirschfeldt.ea:12,Greenberg.Nies:11}. 

\vsp

\n 1.  In \cite[Thm 2.9]{Greenberg.Hirschfeldt.ea:12}  it is shown that given a non-empty $\PPI$ class $P$, each  jump traceable set $A$ Turing below each superlow member of $P$ is already  strongly jump traceable. In particular this applies to  superlow c.e.\ sets $A$, since such sets are jump traceable \cite{Muenster}.  For many  non-empty $\PPI$ classes   such a set is in fact computable. For instance, it  could be a class where any two distinct members form a minimal pair.   In contrast, the nonempty among the  $\PPI$ classes $P=P_{h,d}$ are  examples where being below each superlow (or $\omega$-c.e.) member   characterizes strong jump traceability for c.e.\ sets.

\vsp

 \n 2.  Each superlow set $A$ is weak truth-table below {\it some}  superlow set  $Y$ as in  the hypothesis of Theorem~\ref{thm:Kuc-g}. For let $P$ be the class of $\{0,1\}$-valued d.n.c.\  functions.  By \cite[1.8.41]{Nies:book} there is a set $Z \in P$ such that $(Z \oplus A)' \ltt A'$. Now let $Y = Z \oplus A$. This contrasts with the case of   ML-random covers: if a c.e.\ set $A$ is   not $K$-trivial,   then each ML-random set Turing above $A$ is already Turing above $\Halt$ by \cite{Hirschfeldt.Nies.ea:07}.  Thus,  in the case of {\it ibT} reductions,  Theorem~\ref{thm:Kuc-g} applies to more oracle sets $Y$ than    \cite[Prop.\ 5.2]{Greenberg.Nies:11}.  

\vsp

\n 3. 
Greenberg and Nies \cite[Prop.\ 5.2]{Greenberg.Nies:11}  have shown that   for  each order function~$p$, each  strongly jump traceable c.e.\ set is Turing below below each $\omega$-c.e.\ ML-random set, via a reduction with  use  bounded by $p$.
We could also strengthen Theorem~\ref{thm:Kuc-g} to yield such a  ``$p$-bounded'' Turing reduction. 
%%%%%%%%%%%%%%%%%%%%

\iffalse
	\subsection{More on benign cost functions}

	In the following we strengthen a  result of \cite{Greenberg.Nies:11} that a c.e.\ set obeying all benign cost functions is strongly jump traceable: we remove   the hypothesis that the given set be c.e.\  
	\begin{thm} 
	Suppose a $\DII$ set $A$ obeys all benign cost functions. Then~$A$ is strongly jump traceable. \end{thm}

	\begin{proof} One  can derive the theorem from results in \cite{Greenberg.Hirschfeldt.ea:12}. For each $\omega$-c.e.\ set $Y$, the set  $A$ obeys the benign cost function $\cc_Y$ defined  in (\ref{eqn: cZ}).   Hence $A \leT Y$. By \cite[Thm. 2.1]{Greenberg.Hirschfeldt.ea:12} this implies that $A$ is strongly jump traceable.

	It is instructive to give a direct proof of the theorem, relying  only on  technical result from 
	\cite{Greenberg.Hirschfeldt.ea:12} about  so-called restrained approximations. Since  $A \models \cost$,  $A$  is $K$-trivial. Hence  $A$ is jump traceable and superlow.  

	Given an order function~$h$,  we will define  a jump trace $(T_n) \sN n$  with bound $h$ for~$A$.  Define a cost function~$\cc$ similar to   (\ref{eqn: cZ}):     let  $ \cc(x,s) =0 $ for each $x\ge s$. If  $x<s$ and~$n<x$ is
	 least such that
	$J^A(n) \DA [s]$ but $J^A(n) \UA [s-1]$, then       let $ \cc(x,s) = \max( \cc(x,s-1), 1/h(n) )$.

	Since $(A_s, \wt J_s)\sN s$ is a restrained approximation it is easy to see that~$\cc$ is benign. So  choose a computable enumeration 
	$(\hat  A_r)\sN{r}$   of~$A$
	obeying~$\cc$.

	\end{proof}

\fi

\subsection{A proper implication between cost functions}
In this subsection we study  a weakening of $K$-triviality using the monotonic cost function
\[ \cc_{\max }  (x, s) = \max \{\tp{-K_s(w)} \colon \, x < w \le s \}.\]
Note that $\cc_{\max }$ satisfies  the limit condition, because  \bc $\ul \cc_{\max }(x) =\max \{\tp{-K(w)} \colon \, x < w  \}  $.  \ec

Clearly $\cc_{\max } (x,s) \le \cost(x,s) $, whence  $\cost \ria \cc_{\max }$. We will show that this implication of cost functions is proper.  Thus, some set obeys $\cc_{\max } $ that is not $K$-trivial.

Firstly, we investigate   sets obeying $\cc_{\max }$. 	For a string $\aaa$, let $g(\aaa)$ be the longest prefix of $\aaa$ that ends in $1$, and $g(\aaa)= \estring$ if there is no such prefix.

\begin{definition} {\rm  We say that a set $A$ is \emph{weakly $K$-trivial} if \bc $\fao n [ K(g(A\uhr n))\lep K(n)]$. \ec } \end{definition}

Clearly, every $K$-trivial set is weakly $K$-trivial. By the following, every \emph{c.e.}\ weakly $K$-trivial  set is already $K$-trivial.

\begin{fact} If $A$ is weakly $K$-trivial and  not h-immune, then $A$ is $K$-trivial. \end{fact}
	
	\begin{proof} By the second hypothesis, there is an increasing computable function~$p$ such that $ [p(n), p(n+1)) \cap A \neq \ES $ for each~$n$.  Then 
		\bc $K(A \uhr {p(n)})\lep K(g(A \uhr {p(n+1)}))\lep K(p(n+1))\lep K(p(n))$. \ec
 This implies that $A$ is $K$-trivial by \cite[Ex.\ 5.2.9]{Nies:book}. \end{proof} 
		We say that a computable approximation $\seq {A_s} \sN s$ is \emph{erasing} if for each~$x$ and each  $s>0$, $A_s(x) \neq A_{s-1}(x)$ implies $A_s(y)= 0 $ for each $y$ such that  $x< y \le s$.  For instance, the computable approximation built in the proof of the implication ``$\RA$'' of Theorem~\ref{thm:imply} is erasing by the construction. %,   because when a requirement  $R_e$ acts,  it initializes the lower priority  requirements.

		\begin{prop} \label{prop:CMax_implies_weakly_K-trivial} Suppose $\seq {A_s} \sN s$ is an erasing computable approximation of a set $A$, and $\seq {A_s} \models \cc_{\max }$. Then $A$ is weakly $K$-trivial. \end{prop}
	\begin{proof} This is a modification of the usual proof that every set $A$ obeying $ \cost$ is $K$-trivial (see, for instance, \cite[Thm.\  5.3.10]{Nies:book}). 
		
		To show that  $A$ is weakly $K$-trivial,  one  builds a bounded request set $W$. When at stage $s>0$ we have $r = K_s(n) < K_{s-1}(n)$,   we  put the  request  $\la r+1, g(A\uhr n)\ra$ into $W$. When $A_s(x) \neq A_{s-1}(x)$,    let $r$ be the number such that $\cc_{\max }(x,s)  = \tp{-r}$, and put the request $\la r+1, g(A\uhr {x+1})\ra$ into $W$.
		
		 Since the computable approximation $\seq {A_s} \sN s$ obeys $\cc_{\max }$, the set $W$ is indeed a bounded request set; since $\seq {A_s} \sN s$ is erasing, this bounded request set shows that $A$ is weakly $K$-trivial. 
	\end{proof}

We now prove that 	$\cc_{\max } \not \ria \cost$. We do so via proving a reformulation that is of interest by itself.

	\begin{thm} \label{thm:cmax_cK_separation}  For every $b \in \NN$ there is an $x$ such that $\ulcost(x) \ge \tp b \ul \cc_{\max }(x)$. In other words,  \bc 	$ \sum \{\tp{-K(w)} \colon \, x < w  \}  \ge \tp b \max \{\tp{-K(w)} \colon \, x < w  \}  $. \ec  \end{thm}

		By Theorem~\ref{thm:imply}, the statement of the foregoing Theorem
		is equivalent to  $\cc_{\max } \not \rightarrow \cost$.  Thus, as remarked above, some set $A$  obeys $\cc_{\max }$ via an erasing computable approximation, and does not obey $\cost$. By Proposition~\ref{prop:CMax_implies_weakly_K-trivial} we obtain a separation. 
		
		\begin{cor} Some  weakly $K$-trivial  set fails to be $K$-trivial. \end{cor}
		Melnikov and Nies~\cite[Prop.\ 3.7]{Melnikov.Nies:12} have given an alternative proof of the preceding result by constructing a weakly $K$-trivial set that is Turing complete.
		\begin{proof}[Proof of Theorem~\ref{thm:cmax_cK_separation}] Assume that there is $b \in \NN$ such that \bc  $\fao x  [\ulcost(x) < \tp b \ul \cc_{\max }(x)]$. \ec
			  To obtain a contradiction, the idea  is that $\cost(x, s)$, which is  defined as a sum,  can be made large in many small bits; in contrast,  $\cc_{\max }(x,s)$, which depends  on the value $\tp{-K_s(w)}$ for a single $w$, cannot.
			
			  We will define a sequence $0=x_0 < x_1 < \ldots < x_N$ for a certain   number $N$. When $x_v$ has been defined for $v<N$, for a certain stage $t> x_v$ we cause  $\cost (x_v,t)$ to exceed  a  fixed   quantity  proportional to $1/N$. We wait until  the opponent  responds  at a stage $s > t$   with  some $w> x_v $  such that  $\tp{-K_s(w)}$ corresponding to that quantity. Only then, we  define $x_{v+1}=s$. For us, the cost $\cost(x_i,x_j)$ will accumulate for $i<j$, while the opponent has to provide  a new~$w$ each time. This means that eventually he will run out of space in the domain of the prefix-free machine giving   short descriptions of such  $w$'s.
			
In the formal construction, we will  build a bounded request set $L$ with the purpose to cause $\cost(x,s)$ to be large  when it is  convenient to us. We may assume by the recursion theorem that the  coding constant for $L$ is given in advance (see \cite[Remark 2.2.21]{Nies:book} for this standard argument). Thus, if we put a request $\la n, y+1 \ra$ into $L$ at a stage $y$, there will be a stage $t>y$ such that $K_t(y+1 ) \le n+d$, and hence $\cost(x,t) \ge \cost (x,y) + \tp{-n-d}$. 
	
Let $\k= \tp{b+d+1}$. Let $N= \tp \k$.% Let $\l= \k - b- d$. 

\vsp

\n \emph{Construction of $L$ and a sequence $0=x_0 < x_1 < \ldots < x_N$ of numbers.} 

\vsps 

\n Suppose $v< N$ and $x_v$ has already  been defined. Put $\la \k, x_v+1 \ra$ into $L$. As remarked above, we may wait for a stage $t>x_v$ such that  $\cost(x_v, t)\ge \tp{-\k-d}$. Now, by our assumption,  we have $\ulcost(x_i) < \tp b \ul \cc_{\max }(x_i)$ for each $i \le v$. Hence we can wait for a stage $s> t$ such that 
\begin{equation} \label{eqn:wait_cost_bounded} 
	\fa i \le v \,  \exo w \big [ x_i < w \le s \lland \cost(x_i, s ) \le \tp{b  -K_s(w)}]. 	\end{equation} 
	
	Let $x_{v+1}=s$.  This ends the construction.
	
	\vsp
	\verif Note that $L$ is indeed a bounded request set. Clearly we have $\cost(x_i, x_{i+1}) \ge \tp{-\k-d}$ for each $i<N$. 
	
	\begin{claim} \label{claim:Dali}  Let $r\le \k$. Write $R = \tp r$. Suppose $p+R \le N$. Let $s = x_{p+R}$. Then we have 
	
	\begin{equation} \label{eqn:claim_ineq} \sum_{w = x_p+1}^{x_{(p+R)}}\min (\tp{ - K_s(w)}, \tp{-\k-b-d+r}) \ge (r+1) \tp{-\k-b-d+r-1}.
	\end{equation} 
\end{claim}
\n For $r= \k$, the right  hand side equals $(\k+1) \tp{-(b+d+1)}>1$, which is a contradiction because the left hand side is at most $\Om \le 1$. 

\vsp

We prove the claim by induction on $r$. 
To verify the case $r=0$, note that by (\ref{eqn:wait_cost_bounded}) there is $w \in (x_p, x_{p+1}]$  such that $\cost(x_p, x_{p+1}) \le \tp {b - K_s(w)}$. Since $\tp{-\k-d} \le \cost(x_p, x_{p+1})$,  we obtain \bc  $\tp{-\k - b - d} \le \tp{-K_s(w)}$ (where $s= x_{p+1}$). \ec  Thus the left hand side in the inequality (\ref{eqn:claim_ineq}) is at least $\tp{-\k - b - d}$, while the right hand side equals $\tp{-\k - b - d-1}$, and the claim holds for $r=0$.  

\vsps

In the following, for $i< j \le N$,  we will write $\SS(x_i,x_j)$ for a sum of the type occurring in (\ref{eqn:claim_ineq}) where $w$ ranges from $x_i+1$ to $x_j$.

Suppose inductively the claim has been established   for $r<\k$.  To verify  the claim for $r+1$, suppose  that $p + 2R \le N$ where $R = \tp r$ as before. Let $s= x_{p+2R}$.  Since $\cost(x_i, x_{i+1}) \ge \tp{-\k-d}$, we have 

\bc $\cost(x_p,s) \ge 2R \tp{-k-d} = \tp{-\k-d+r+1}$. \ec 
By   (\ref{eqn:wait_cost_bounded}) this implies that there is $w$, $x_p< w \le s$, such that 
\begin{equation} \label{eqn:hurray_w}\tp{-\k - b - d+r+1} \le \tp{-K_s(w)}. \end{equation}
Now, in  sums of the form  $\SS (x_q,x_{q+R})$, because of taking the minimum,  the ``cut-off'' for how much $w$  can contribute  is at $\tp{-\k-b-d+r}$. Hence we have 
\bc $\SS(x_p, x_{p+2R}) \ge \tp{-\k-b-d+r} + \SS(x_p,x_{p+R}) + \SS(x_{p+R}, x_{p+2R}) $. \ec
 The additional  term $\tp{-\k-b-d+r}$ is due to the fact that $w$ contributes at most $\tp{-\k-b-d+r}$ to $ \SS(x_p,x_{p+R}) + \SS(x_{p+R}, x_{p+2R})$, but by (\ref{eqn:hurray_w}),   $w$  contributes  $\tp{-\k-b-d-r+1}$  to $\SS(x_p, x_{p+2R})$.  By the inductive hypothesis, the right hand side is at least 
\bc $\tp{-\k-b-d+r} + 2 \cdot (r+1) \tp{-\k-b-d+r-1}= (r+2) \tp{-\k-b-d+r}$, \ec  as required.
\end{proof}

\section{A cost function-related  basis theorem for $\PPI$ classes}  \label{s:costf_basis_theorem}
The following strengthens \cite[Thm 2.6]{Greenberg.Hirschfeldt.ea:12}, which relied on  the extra  assumption   that  the $\PPI$ class is contained in the ML-randoms. 
\begin{thm} Let $\P$ be a nonempty $\PPI$ class, and let~$\cc$ be a monotonic  cost function with the limit condition. Then there is a $\DII$ set $Y \in \P$ such that each c.e.\ set $A \leT Y$  obeys $\cc$. 
\end{thm}

\begin{proof} We may assume    that $ \cc(x, s) \ge \tp{-x}$ for each $x\le s$,  because any c.e.\ set that obeys $\cc$ also obeys the cost function  $ \cc(x,s)+ \tp{-x}$. 

Let $\seq{A_e, \Psi_e}\sN e$ be an effective listing of all pairs consisting of a c.e.\ set and a  Turing functional.  We will define a $\DII$ set $Y \in \P$ via a computable approximation ${Y_s} \sN s$,  where~$Y_s$ is a binary string of length $s$. We meet the requirements

\bc $N_e\colon \, A_e = \Psi_e(Y) \RRA A_e $ obeys $ \cc$. \ec
We use  a standard  tree construction at the $\ES''$ level.  Nodes on the tree $\strcantor$  represent the strategies. Each node  $\aaa$ of length~$e$ is a strategy for $N_e$.  At stage $s$ we define an approximation $\delta_s$ to the true path. 
We say that $s$ is an \emph{$\aaa$-stage} if $\aaa \prec \delta_s$.

Suppose that a strategy $\aaa$ is on the true path.  If $\aaa 0 $ is on the true  path,  then strategy~$\aaa$ is able to build a computable enumeration of $A_e$ via which $A_e$ obeys~$\cc$.  If $\aaa 1$ is on the true path, the  strategy shows that $A_e \neq  \Psi_e(Y)$. 

Let $\P^\estring $ be the given class $\P$.   A strategy $\aaa$ has as an environment a $\PPI$ class $\P^\aaa$. It defines  $\P^{\aaa 0} = \P^\aaa$, but usually let   $\P^{\aaa 1}$ be  a proper refinement of $\P^{\aaa}$.

Let $\aaal =e$. The length of agreement for $e$ at  a stage $t$ is $\min\{ y \colon  \, A_{e,t} (y) \neq  \Psi_{e,t} (Y_t)\}$. We say that an  $\aaa$-stage $s$ is $\aaa$-\emph{expansionary} if the length of agreement  for   $e$  at  stage~$s$ is larger than at  $u$ for all previous $\aaa$-stages $u $.

Let $w^n_0 = n$, and
 \begin{equation} \label{eqn:wi} w^n_{i+1} \simeq \mu v > w^n_i .  \,  \cc(w^n_i, v) \ge 4^{-n}. \end{equation}  
Since $\cc$ satisfies the limit condition, for each $n$ this sequence breaks off. 

Let $a= w^n_i$ be  such a value. The basic idea is to {\it certify} $A_{e,s} \uhr w$,  which  means to ensure that all $X \succ Y_s\uhr {n+d}$ on  $\PP^\aaa$ compute $A_{e,s} \uhr w$. If  $A \uhr w$ changes later  then also $Y\uhr {n+d}$ has to change. Since $Y\uhr {n+d}$ can only move to the right (as long as $\aaa$ is not initialized),  this type of change for $n$  can only contribute a cost of  $4^{-n+1} \tp{n+d} = \tp{-n+d+2}$.

By \cite[p.\ 55]{Nies:book},    from an index $\Q$ for a $\PPI$ class  in $\cantor$  we can obtain a computable  sequence $(\Q_s) \sN s$ of clopen classes such that $\Q_s \supseteq \Q_{s+1}$ and $\Q = \bigcap_s \Q_s$. In the construction below we will have several indices for $\QI$ classes $\Q$ that change over time. At stage $s$, as usual by $\Q[s]$ we denote the value of the index at stage $s$. Thus $(\Q[s])_s$ is the clopen approximation of $\Q[s]$  at  stage $s$.

\vsps

\n {\em Construction of} $Y$.

\n \emph{Stage $0$. Let $\delta_0 = \estring$ and $\P^\estring = \P$. Let $Y_0 = \estring$.}

\n \emph{Stage $s>0$.}  Let $\P^\estring = \P$. 

\n For each $\beta$ such that  $\delta_{s-1}  <_L \beta$ we initialize strategy~$\beta$.  We let $Y_s$ be the leftmost path on the current approximation to $\P^{\delta_{s-1}}$, i.e., the  leftmost string $y$ of length $s-1$ such that $[y ] \cap (\P^{\delta_{s-1}}[s-1])_{s} \neq \ES$. For each $\aaa, n$, if $Y_s \uhr {n+d} \neq Y_{s-1} \uhr {n+d} $ where $d= \init_s(\aaa) $, then we declare each existing value $w^n_i$ to be $(\aaa, n)$-\emph{unsatisfied}.

\n \emph{Substage $k$}, $0\le k <  s$. Suppose we  have already defined $\aaa = \delta_s \uhr k$.  Run strategy~$\aaa$ (defined below) at stage $s$, which defines an outcome $r \in \twoset$ and a $\PPI$ class $\P^{\aaa r}$. Let $\delta_s(k) = r$.

We now describe the strategies $\aaa$ and procedures  $\SS^\aaa_n$  they call. To initialize a strategy $\aaa$ means to  cancel the  run of this  procedure.   Let  

\bc $d= \init_s(\aaa)  = \aaal + $the last stage when $\aaa$ was initialized.  \ec

\n \emph{Strategy $\aaa$ at an $\aaa$-stage $s$.}

\n \bi \item[(a)] If no procedure for $\aaa$ is running, call procedure $\SS^\aaa_n$ with parameter~$w$,  where~$n$ is least, and $i$ is chosen least for $n$,  such that    $w= w^n_i\le s$  is not $(\aaa, n)$-satisfied.  Note that $n$ exists because $w^s_0 =s$ and this value is not  $(\aaa, n)$-satisfied at the beginning of stage $s$. 
By calling this procedure,     we attempt to certify  $A_{e,s}  \uhr w$ as discussed above. 
  % If $n$ fails to exist,   give outcome  $  1$,   and let $\P^{\aaa 1}= \P^\aaa$. 

\item[(b)] While such a  procedure $\SS^\aaa_n$ is running,  give outcome  $  1$.  

\n (This procedure  will  define  the current  class $\P^{\aaa 1}$.)  

\item[(c)] If a   procedure $\SS^\aaa_n$ returns at this stage, goto (d).

\item[(d)]  If $s$ is  $\aaa$-expansionary, give outcome $0$,  let $\P^{\aaa 0}= \P^\aaa$,  and continue at  (a) at the next $\aaa$-stage. 
 Otherwise,  give outcome $ 1$, let   $\P^{\aaa 1}= \P^\aaa$, and stay at (d). \ei

\n \emph{Procedure $\SS^\aaa_n$ with parameter~$w$ at a stage $s$.}

\n If $n+d \ge s-1$ let $\P^{\aaa 1} = \P^\aaa$. Otherwise, let 
\begin{equation} \label{eqn:Q}   \Q  = \P^\aaa \cap \{X \succ z \colon \, \Psi^X_e \not \succ A_{e,s}\uhr w\}, 
\end{equation}
where $z = Y_{s}\uhr {n + d}$. (Note  that each time  $Y\uhr {n + d}  $ or  $A_{e}\uhr w$ has  changed, we update this definition of $\Q$.)

\bi \item[(e)] If  $\Q_s  \neq \ES$ let $\P^{\aaa 1} = \Q$. If the definition of $\P^{\aaa 1}$ has changed since the last $\aaa$-stage, then 
  each $\beta$ such that  $\aaa 1 \preceq \beta$   is  initialized.

\item[(f)] If $\Q_s= \ES$, declare $w$ to be \emph{$(\aaa, n)$-satisfied} and return. ($A_{e,s}  \uhr w$ is certified as every $X \in \P^\aaa$ extending $z$ computes $A_{e,s}  \uhr w$  via $\Psi_e$. If $A_e \uhr w$ changes later,  the necessarily $z \not \preceq Y$.)
\ei
\vsps

\begin{claim}   \label{claim:1} 
Suppose  a strategy $\aaa$ is no longer  initialized after stage $s_0$. Then for each $n$,  a  procedure $\SS^\aaa_n$ is only called finitely many times after $s_0$.
\end{claim} 
   There are only finitely many values $w=w^n_i$ because $\cc$ satisfies the limit condition.   Since $\aaa$ is not initialized after $s_0$,  $\P^\aaa$ and $d= \init_s(\aaa) $ do not change. When a run of  $\SS^\aaa_n$ is  called at a stage~$s$, the strategies $\beta \succeq \aaa1 $ are initialized, hence $\init_t(\beta) \ge s > n+d$ for all $t \ge s$.   
   So the string  $Y_s\uhr {n+d}$   is the leftmost string of length $n+d$  on $\P^\aaa$ at stage $s$. This string  has to  move to the right between the stages when  $\SS^\aaa_n$ is  called with the same parameter $w$, because  $w$  is declared $(\aaa, n)$-unsatisfied before $\SS^\aaa_n$ is called again with parameter $w$.   Thus,  procedure $\SS^\aaa_n$ can only be called $\tp{n+d}$ times with parameter $w$. 

\begin{claim} \label{claim:2}  $\seq {Y_s} \sN s$ is a computable approximation of a $\DII$ set $Y \in \P$. \end{claim} 

Fix $k\in \NN$. For a stage $s$,    if $Y_s\uhr k $ is to the left of $ Y_{s-1} \uhr k$ then there are $\aaa, n$ with  $n + \init_s(\aaa) \le k$ such that   $\P^\aaa[s] \neq \P^\aaa [{s-1}]$ because of the action of a procedure $\SS^\aaa_n$ at (e) or (f). 

There are only finitely many pairs $\aaa, s$  such that  $ \init_s(\aaa) \le k$. Thus by Claim~\ref{claim:1} there is stage $s_0$ such that at all stages   $s\ge s_0$, for no     $\aaa  $    and   $n$ with $ n+ \init_s(\aaa) \le k$, a  procedure $\SS^\aaa_n$ is called. 

While a procedure $\SS^\aaa_n$ is running with a parameter $w$, it  changes the definition of   $\P^{\aaa 1}$ only if $A_e \uhr w$ changes ($e = \aaal$), so at most $w$ times. Thus there are  only finitely many $s$ such that $Y_s\uhr k \neq  Y_{s-1} \uhr k$. 

By the definition of the computable approximation  $\seq {Y_s} \sN s$ we have $Y \in \P$. This completes Claim~\ref{claim:2}. 

As usual,  we define the true path $f$   by $f(k) = \liminf_s \delta_s(k)$. By Claim~\ref{claim:1}   each $\aaa \prec f$ is only initialized finitely often, because each $\beta$ such that $\beta 1 \prec \aaa$ eventually is stuck with a single   run of a procedure  $\SS^\beta_m$. 

\begin{claim} If $e = \aaal $ and $\aaa 1 \prec f$, then $A_e \neq \Psi_e^Y$. \end{claim}

Some procedure $\SS^\aaa_n$ was called with parameter $w$, and  is eventually stuck at (e) with the final value
$A_e\uhr w$. Hence the definition $\Q= \PP^{\aaa1}$    eventually  stabilizes  at $\aaa$-stages $s$. Since $Y \in \Q$,  this implies $A_e \neq \Psi_e^Y$. 

\begin{claim}   \label{claim:4} 
If  $e = \aaal $ and $\aaa 0 \prec f$, then $ A_e$ obeys $\cc$.
\end{claim} 
Let  $A= A_e$. We define a computable enumeration $(\hat A_p) \sN p$ of $A$ via which $A$ obeys $\cc$. 

Since $\aaa 0 \prec f$, each procedure $\SS^\aaa_n$ returns. In particular, since $\cc$ has the limit condition and by Claims~\ref{claim:1} and~\ref{claim:2}, each value $w= w^n_i$   becomes permanently  $(\aaa, n)$-satisfied. Let $d= \init_s(\aaa)$.   Let $s_0$ be the least  $\aaa0$-stage  such that $s_0 \ge d$, and let

\vsps

   $s_{p+1} = \mu s \ge  s_p+2  \,  [ s \ttext{is $\aaa0$-stage} \lland $
  \bc $\forall n, i \,  ( w = w^n_i <  s_p  \ria \,  w \ttext{is $(\aaa, n)$-satisfied at} s ) ]$. \ec
As in similar constructions such as~\cite{Nies:book},  for $p \in \NN$ we let \bc $\hat A_p = A_{s_{p+2}} \cap [0,p)$. \ec

  Consider  the situtation that $p > 0$ and $x\le p$ is least  such that  $\hat A_p(x) \neq \hat A_{p-1}(x)$. We call this situation  an  \emph{$n$-change} if $n$ is  least such that  $x<  w^n_i < s_p$ for some $i$. (Note that $n\le p+1$ because $w_0^{p+1} = p+1$.) Thus $(x,s_p)$ contains no value of the form $w^{n-1}_j$, whence  $ \cc(x, p) \le  \cc(x,s_p)  \le 4^{-n+1}$.  We are done if we can show there are at most $\tp{n+d}$ many $n$-changes, for in that case  the total cost $\cc  \seq {\hat A_p } $  is bounded by $\sum_n 4^{-n+1} \tp{n+d} =O(2^d)$.

 Recall that $\P^\aaa$ is stable by stage $s_0$.   Note that   $Y\uhr{n+d}$ can only move to the right after the first run of $\SS^\aaa_n$, as observed in the proof of Claim~\ref{claim:1}.  
 
 Consider $n$-changes at stages $p<q$ via parameters $w = w^n_i$ and  $w'= w^n_k$ (where possibly $k<i$). Suppose the last run of $\SS^\aaa_n$ with parameter $w$ that was  started before $s_{p+1}$  has returned at stage $t \le s_{p+2}$, and similarly,  the  last run of $\SS^\aaa_n$ with parameter $w'$ that was  started before $s_{q+1}$  has returned at stage $t'$. Let $z= Y_t \uhr {n+d}$ and  $z'= Y_{t'} \uhr {n+d}$. We   show $z <_L z'$; this implies that there are at most $\tp{n+d}$ many $n$-changes. 
  
  At stage $t$,  by definition of returning at (f) in the run of $\SS^\aaa_n$, we have $\Q = \ES$. Therefore $ \Psi^X_{e,t}  \succ A_{e,t}\uhr w$ for each $X$ on $   \P^\aaa_t$ such that  $ X \succ z $. Now  \bc $\hat  A_p(x)  \neq \hat A_{p-1}(x)$, $x<w$   and $t \le s_{p+1}$,  \ec so  $A_{s_{p+2}} \uhr w \neq A_t \uhr w$, The stage  $s_{p+2}$ is $\aaa0$-expansionary, and $Y_{s_{p+2}}$ is on  $\P^\aaa_t$. Therefore
    \bc $Y_{r-1} \uhr {n+d}\,  <_L Y_{r} \uhr {n+d}$  \ec
  for some stage $r$ such that  $t < r \le s_{p+2}$. Thus,  at stage $r$, the value $w'$ was declared $(\aaa, n)$-unsatisfied.  Hence a new run of $\SS^\aaa_n$ with parameter $w'$ is  started after $r$, which   has returned by stage $s_{q+1} \ge s_{p+2}$. Thus $r< t'$.  So   $z \le_L Y_{r-1} \uhr {n+d}  <_L Y_{r} \uhr {n+d}   \le_L  z'$, whence $z <_L z'$  as required. 
 This concludes   Claim~\ref{claim:4}  and the proof. 
 \end{proof}
%%%%%%%%%%%%%%%%%%%%%%%%%%%%%%

\section{A dual cost function construction} \label{ss:dual}
Given a relativizable cost function $\cc$, let $D \rightarrow W^D$ be the c.e.\ operator given by the cost function construction   in  Theorem~\ref{thm:cfconstr}  relative to  the  oracle~$D$. By pseudo-jump inversion there is a c.e.\ set $D$ such that $W^D \oplus  D \equiv_T \Halt$, which implies $D <_T \Halt$.   
Here, we give a direct construction of  a c.e.\ set $D <_T \Halt$ so  that  the total cost of $\Halt$-changes   as measured by $\cc^D$ is finite. More precisely, there is a $D$-computable enumeration of $\Halt $ obeying~$\cc^D$. 

If $\cc$ is sufficiently strong, then 
  the   usual  cost function construction    builds an incomputable c.e.\   set $A$ that is close to being computable. The dual cost function construction then builds a c.e.\ set $D$ that is close to being Turing complete. 
  
 \subsection{Preliminaries on cost functionals}  Firstly we clarify  how to relativize  cost functions, and the notion of obedience to a cost function. Secondly we provide some technical details needed for the main construction. 
\begin{deff}  {\rm (i)  A \emph{cost functional}  is a Turing functional $\cc^Z(x,t)$ such that for each oracle $Z$, $\cc^Z$ either  is partial,  or  is a cost function relative to $Z$. We say that $\cc$ is non-increasing  in main argument    if this holds for each oracle $Z$ such that $\cc^Z $ is total. Similarly, $\cc$ is non-decreasing  in the stage argument    if this holds for each oracle $Z$ such that $\cc^Z $ is total. If both properties hold we say that $\cc$ is monotonic.

\n (ii) Suppose  $A \leT Z'$. Let $\seq{A_s}$ be a  $Z$-computable approximation of $A$. 

\n We write $\seq{A_s}  \models^Z \cc^Z$ if    

\vsps

  $ {\cc^Z}   {\seq{A_s}} = \sum_{x,s}  \cc^Z(x,s) $

\hfill $
 \Cyl{ x< s  \lland   
 \cc^Z(x,s) \DA \lland  x \ttext{ is
least s.t.}   A_{s-1} ( x) \neq  A_{s} (x) } $ 

\vsps

\n  is finite. 
We write $A \models^Z c^Z$ if $\seq{A_s}  \models^Z c^Z$ for some $Z$-computable approximation  $\seq{A_s}$ of $A$. } \end{deff} 

\n For example, $\cost^Z(x,s) = \sum_{x < w \le s} 2^{-K^Z_s(w)}$ is a total  monotonic cost functional.   We have $A\models^Z \cost^Z$ iff $A $ is $K$-trivial relative to $Z$. 

%Another example: cost function for s.j.t.
%
%We can always require totality for all oracles...

\vsp

 We may convert a cost functional $\cc$ into a  total cost functional $\wt \cc$  such that $\wt \cc^Z(x)= \cc^Z(x)$ for each $x$ with   $\fao t \cc^Z(x,t)\DA$, and, for each $Z,x,t$, the computation  $\wt \cc^Z(x,t)$ converges in $t$ steps.     Let 
 
 \bc $\wt \cc^Z(x,s) = \cc^Z(x,t)$ where $t\le s$ is largest such that $\cc^Z(x,t)[s] \DA$. \ec Clearly, if  $\cc$ is monotonic in the main/stage argument then so is $\wt \cc$.

Suppose that $D$ is c.e.\ and we   compute $\cc^D(x,t)$  via    hat computations \cite[p.\ 131]{Soare:87}: the use  of a computation $\cc^D(x,t)[s]\DA$   is no larger  than the least number entering $DÄ$ at stage $s$. Let $N_D $ be the set of non-deficiency stages; that is, $s \in N_D$ iff there is  $x\in D_s- D_{s-1}$ such that $D_s\uhr x = D\uhr x$. Any hat computation existing at a non-deficiency stage is final. We have
\begin{equation} \label{hat cost} \cc^D(x) = \sup_{s \in N_D} \wt \cc^{D_s}(x,s). \end{equation}
For, if $\cc^D(x,t)[s_0] \DA$ with $D$ stable below the use,  then $\cc^D(x,t) \le \wt \cc^{D_s}(x,s)$ for each $s \in  N_D$. Therefore $ \cc^D(x) \le  \sup_{s \in N_D} \wt \cc^{D_s}(x,s)$. For the converse inequality, note that for $s \in N_D$ we have $\wt \cc^{D_s}(x,s) = \cc^D(x,t)$ for some $t \le s$ with $D$ stable below the use.

% Example: cost function $\cc$ for being in $\omega$-c.e.$^\Diamond$. Makes a T-incomplete c.e. set $D$ such that $\Halt\le_{SJT} D$. 

%In the following we will assume w.l.o.g.\  that (\ref{hat cost}) already holds for the given cost functional $\cc$. 

\subsection{The  dual  existence theorem} 

\begin{thm} Let $\cc$ be a  total cost functional that is nondecreasing in the stage component and satisfies the limit condition for each oracle. Then there is a   Turing incomplete c.e.\ set $D$ such that $\Halt \models^D \cc^D$. \end{thm}

\begin{proof}  We define a cost functional $\Gamma^Z(x,s)$ that is nondecreasing in the stage. We will have $\ul \Gamma^D(x) = \cc^D(x)$ for each $x$, where $\ul \Gamma^D(x) = \lim_t \Gamma^D(x,t)$, and $\Halt$ with its given  computable enumeration obeys $\Gamma^D$.   Then $\Halt \models^D \cc^D$ by the easy direction `$\LA$' of Theorem~\ref{thm:imply} relativized to $D$.

Towards  $ \Gamma^D(x) \ge \cc^D(x)$, when we see a computation $\wt \cc^{D_s}(x,s) = \aaa$ we attempt to ensure that   $\Gamma^D(x,s) \ge \aaa$.
To do so we enumerate relative to $D$ a set $G$ of  ``wishes''  of the form 
\bc $\rho = \la x, \aaa \ra^u$, \ec
where $x\in \NN$,  $\aaa$ is a nonnegative rational, and $u+1$ is the use. We say that $\rho$ is a {\it wish about  $x$}.  If such a wish  is enumerated at a stage $t$ and $D_t \uhr u$ is stable,  then  the wish is granted, namely,  $\Gamma^D(x,t) \ge \aaa$.  The converse inequality $ \Gamma^D(x) \le \cc^D(x)$ will hold automatically.

To ensure $D <_T \Halt$, we enumerate a set $F$,  and meet the requirements 
\bc $N_e \colon \, F  \neq \Phi_e^D$. \ec
Suppose we have put  a wish  $\rho = \la x, \aaa \ra^u$ into $G^D$. To keep  the total $\Gamma^D$-cost of the given computable enumeration of $\Halt$ down, when $x$ enters $\Halt$ we   want to remove $\rho$ from $G^D$   by putting $u$ into $D$. However, sometimes $D$ is preserved by some $N_e$. This will generate a {\it preservation cost}. $N_e$ starts a run at a stage $s$  via some parameter $v$, and ``hopes'' that $\Halt_s \uhr v$ is stable. If $\Halt \uhr v$ changes after stage  $s$, then this run of $N_e$ is cancelled. On the other hand, if $x\ge v$ and $x $ enters $\Halt$, then the ensuing preservation cost can be afforded. This is so because we choose $v$ such that $\wt c_s^{D_s}(v,s) $ is  small. Since $\wt \cc^D$ has the limit condition, eventually there is a run $N_e(v)$ with such a low-cost $v$ where $\Halt \uhr v$ is stable. Then the diagonalization of $N_e$ will succeed. 

\vsp

\n {\it Construction of c.e.\ sets $F,D$ and a  $D$-c.e.\ set $G$ of wishes.}

\n {\it Stage $s > 0$. } We may suppose that there is a unique $n \in \ES'_s - \ES'_{s-1}$. 

\vsps

\n {\it 1.\ Canceling $N_e$'s.} Cancel all currently active $N_e(v)$ with  $v > n$. 

\n {\it 2. Removing wishes.} For each $\rho= \la x, \aaa \ra^u\in G^D[s-1]$ put in at a stage $t<s$, if $  \Halt_s \uhr {x+1} \neq \Halt_{t} \uhr {x+1}$ and    $\rho$ is not held by any $N_e(v) $, then put $u-1$ into $D_s$, thereby removing $\rho $ from $G^D$. 

\n {\it 3. Adding wishes.} For each  $x< s$ pick a large  $u$ (in particular, $u \not \in  D_s$) and put a wish $\la x, \aaa \ra ^u$ into $G$ where $\aaa = \wt \cc^{D_s}(x,s)$. The set of  queries to the oracle $D$  for this enumeration into $G$ is  contained in $[0,r) \cup \{ u\}$, where $r $ is the  use of $  \wt \cc^{D_s}(x,s)$ (which may be much smaller than $s$).    Then, from now on this wish is kept in $G^D$
unless  (a) $D\uhr r$ changes , or  (b)  $u$ enters $D$. 

\n {\it 4. Activating  $N_e(v)$.}  For each $e < s$  such that  $N_e$ is not  currently active, see if  there is $v$, $e \le v \le n$ such that  
\bi \item[--] $\wt \cc^{D_s}(v,s)  \le 3^{-e}/2$,
\item[--]  $v> w$ for each $w$ such that   $N_i(w)$ is active for some   $i< e$, and 
\item[--]   $\Phi_e^D \uhr{x+1}  =F \uhr {x+1} $ where $x =  \la e, v, |\Halt \cap [0,v)|\ra$, 
 \ei
If so,  choose $v$ least and activate  $N_e(v)$. Put $x $ into $F$. Let  $N_e$  {\it hold} all wishes  for some $y \ge v$   that are   currently in $G^D$. Declare that such a wish is no longer  held by any $N_i(w)$ for $i\neq e$. (We also say that $N_e$ {\it takes over} the wish.)

Go to stage $s'$ where $s'$ is larger than any number mentioned so far.

 \vsp

\n {\bf Claim 1.} {\it  Each requirement $N_e$ is activated only finitely often, and  met. Hence $F \not \leT D$.}

\n  Inductively suppose  that $N_i$ for   $i<e$ is no longer activated after stage $t_0$. Assume for a contradiction that $F= \Phi_e^D$. Since $\cc^D$ satisfies  the limit condition, by (\ref{hat cost}) there is a least $v$ such that $\wt \cc^{D_s}(v,s)\le 3^{-e}/2$ for infinitely many $s> t_0$. Furthermore, $v> w$ for any $w$ such that some $N_i(w)$,   $i< e$,  is active at $t_0$.  Once $N_e(v) $ is activated,  it can only be canceled by a change of  $\Halt \uhr v$. Then there is   a stage $s> t_0$, $\wt \cc^{D_s}(v,s)\le 3^{-e}/2$,   such that  $\Halt \uhr  v$ is stable at $s$ and $\Phi_e^D \uhr{x+1}  =F \uhr {x+1} $ where $x =  \la e, v, |\Halt \cap [0,v)|\ra$.  If some $N_e(v')$ for $v' \le v$ is active after  (1.) of stage $s$  then it remains active, and  $N_e$ is met. Now suppose otherwise.

Since  we do not  activate $N_e(v)$ in (4) of stage $s$,  some  $N_e(w)$ is active for   $w> v$. Say it was activated last at a stage $t< s$ via $x= \la e,w, | \ES'_t \cap [0,w]|$.  Then $x' =  \la e, v, |\Halt_t \cap [0,v)|\ra$ was available to activate $N_e(v)$ as $x' \le x$ and hence  $\Phi_e^D \uhr {x'+1} = F \uhr {x'+1} [t]$. Since $w$ was chosen minimal for $e$ at stage $t$,   we had 
  $\wt \cc^{D_t}(v,t) > 3^{-e}/2$. On the other hand,  $\wt \cc^{D_s}(v,s)\le 3^{-e}/2$, hence $D_t \uhr t \neq D_s\uhr t$.   When $N_e(w)$ became active at $t$  it tried to preserve $D\uhr t$ by holding all wishes  about  some $y \ge w$ that were     in $G^D[t]$. Since $N_e(w)$ did   not succeed, it was cancelled by a change   $\Halt_t \uhr w \neq \Halt_s \uhr w$. Hence $N_e(w)$ is not active at stage $s$, contradiction.   \hfill $\Diamond$

\vsps

We now define $\Gamma^Z(x,t)$ for an oracle $Z$ (we are  interested only in the case that  $Z=D$). Let $s$ be least such that $D_s \uhr t = Z\uhr t$. Output the maximum $\aaa $ such that some   wish  $\la x, \aaa \ra^u  $ for $u \le t$ is in $G^D [s]$.

 \vsp
\n {\bf Claim 2.} {\it  (i) $\Gamma^D(x,t)$ is nondecreasing  in $t$. (ii)   $\fao x \ul \Gamma^D(x) = \ul  \cc^D(x)$.}

\n (i).  Suppose  $t' \ge t$. As above let $s$ be least such that $D_{s} \uhr {t}$ is stable.  Let $s'$ be least such that $D_{s'} \uhr {t'}$ is stable. Then $s' \ge s$, so    a wish as in the definition of $\Gamma^D(x,t)$ above  is also in $G^D[s']$. Hence $\Gamma^D(x,t') \ge \Gamma^D(x,t)$.

\n (ii). Given $x$, to show that $\Gamma^D(x) \ge \cc^D(x)$ pick $t_0$ such that $  \Halt \uhr {x+1}$ is stable at $t_0$.  Let $s \in N_D$ and $s > t_0$. At stage $s$ we put a wish $\la x, \aaa\ra^u$ into $G_D$ where $\aaa= \wt \cc^{D_s}(x,s)$. This wish is not removed later, so $\Gamma^D(x) \ge \aaa$.

For  $\Gamma^D(x) \le \cc^D(x)$, note that for each $s \in N_D$ we have
  $\wt \cc^{D_s}(x,s) \ge \Gamma^{D_s}(x,s)$ by the removal of  a wish  in 3(a) of the construction when the reason the wish was there disappears. 
\hfill $\Diamond$

\vsps

\n {\bf Claim 3.} {\it     The  given computable enumeration of $\Halt $  obeys $\Gamma^D$. }

\n First we show by induction on stages $s$ that {$N_e$ holds in total  at most $3^{-e}$ at the end of stage $s$, namely, 

\begin{equation} \label{heyyy}  3^{-e} \ge \sum_x \max \{ \aaa \colon \, N_e \ttext{holds a wish} \la x, \aaa \ra^u\} \end{equation}

Note that once $N_e(v)$ is activated and holds some wishes, it will not hold any further  wishes later,  unless it is cancelled by a change of $\Halt \uhr v$  (in which case the wishes it holds are removed).

We may assume that $N_e(v)$ is activated at (3.) of  stage $s$. Wishes held at stage $s$ by  some $N_i(w)$ where  $ i <  e$ will not be taken over by $N_e(v)$ because $w<v$.  Now  consider wishes held by a  $N_i(w)$ where  $ i >  e$. By inductive hypothesis  the total of such wishes  is at most $\sum_{i > e} 3^{-i} = 3^{-e}/2 $   at the beginning of stage $s$.   The activation of $N_e(v)$ adds at most another $3^{-e}/2 $ to the sum in (\ref{heyyy}).

To show $\Gamma^D \seq  {\Halt_s} < \infty$, note that any contribution to this quantity due to $n$ entering $\Halt $ at stage $s$ is because a  wish $\la n, \delta \ra^u $ is eventually  held by some $N_e(v) $. The total is at most $\sum_e 3^{-e}$.   }
\end{proof}

The study of non-monotonic cost function is left to the future. For instance, we conjecture that there are cost functions $\cc, \dd$ with the limit condition such that for any $\DII$ sets $A,B$,  
\bc $A \models \cc$ and $B \models \dd$ $\RA$ $A,B$ form a minimal pair. \ec
It is not hard to build cost functions $\cc,\dd$ such that only computable sets obey  both of them. This provides some evidence for the conjecture.

%\bibliographystyle{plain}
%
%\bibliography{../Logicsharing/bibs/Nies,../Logicsharing/bibs/randomness,../Logicsharing/bibs/various,../Logicsharing/bibs/recursiontheory}

\begin{thebibliography}{10}

\bibitem{Bienvenu.Day.ea:14}
L.~Bienvenu, A.~Day, N.~Greenberg, A.~Ku{\v{c}}era, J.~Miller, A.~Nies, and
  D.~Turetsky.
\newblock Computing {$K$}-trivial sets by incomplete random sets.
\newblock {\em Bull. Symb. Logic}, 20:80--90, 2014.

\bibitem{Bienvenu.Downey.ea:15}
L.~Bienvenu, R.~Downey, A.~Nies, and W.~Merkle.
\newblock Solovay functions and their applications in algorithmic randomness.
\newblock {\em Journal of Computer and System Sciences}, 81(8):1575--1591,
  2015.

\bibitem{Bienvenu.Greenberg.ea:16}
L.~Bienvenu, N.~Greenberg, A.~Ku{\v{c}}era, A.~Nies, and D.~Turetsky.
\newblock Coherent randomness tests and computing the {K}-trivial sets.
\newblock To appear in J. European Math. Society, 2016.

\bibitem{Calude.Grozea:96}
C.~Calude and A.~Grozea.
\newblock The {K}raft-{C}haitin theorem revisited.
\newblock {\em J. Univ. Comp. Sc.}, 2:306--310, 1996.

\bibitem{Day.Miller:14}
A.~R. Day and J.~S. Miller.
\newblock Cupping with random sets.
\newblock {\em Proceedings of the American Mathematical Society},
  142(8):2871--2879, 2014.

\bibitem{Diamondstone.Greenberg.ea:nd}
D.~Diamondstone, N.~Greenberg, and D.~Turetsky.
\newblock Inherent enumerability of strong jump-traceability.
\newblock {S}ubmitted, available at http://arxiv.org/abs/1110.1435, 2012.

\bibitem{Downey.Hirschfeldt:book}
R.~Downey and D.~Hirschfeldt.
\newblock {\em Algorithmic randomness and complexity}.
\newblock Springer-Verlag, Berlin, 2010.
\newblock 855 pages.

\bibitem{Downey.Hirschfeldt.ea:03}
R.~Downey, D.~Hirschfeldt, A.~Nies, and F.~Stephan.
\newblock Trivial reals.
\newblock In {\em Proceedings of the 7th and 8th Asian Logic Conferences},
  pages 103--131, Singapore, 2003. Singapore University Press.

\bibitem{Figueira.ea:08}
S.~Figueira, A.~Nies, and F.~Stephan.
\newblock Lowness properties and approximations of the jump.
\newblock {\em Ann. Pure Appl. Logic}, 152:51--66, 2008.

\bibitem{Greenberg.Hirschfeldt.ea:12}
N.~Greenberg, D.~Hirschfeldt, and A.~Nies.
\newblock Characterizing the strongly jump-traceable sets via randomness.
\newblock {\em Adv. Math.}, 231(3-4):2252--2293, 2012.

\bibitem{Greenberg.Nies:11}
N.~Greenberg and A.\ Nies.
\newblock Benign cost functions and lowness properties.
\newblock {\em J. Symbolic Logic}, 76:289--312, 2011.

\bibitem{Greenberg.Turetsky:14}
N.~Greenberg and D.~Turetsky.
\newblock Strong jump-traceability and {D}emuth randomnesss.
\newblock {\em Proc. Lond. Math. Soc.}, 108:738--779, 2014.

\bibitem{Hirschfeldt.Nies.ea:07}
D.~Hirschfeldt, A.~Nies, and F.~Stephan.
\newblock Using random sets as oracles.
\newblock {\em J. Lond. Math. Soc. (2)}, 75(3):610--622, 2007.

\bibitem{Kjos.ea:2005}
B.~Kjos-Hanssen, W.~Merkle, and F.~Stephan.
\newblock Kolmogorov complexity and the {R}ecursion {T}heorem.
\newblock In {\em STACS 2006}, volume 3884 of {\em Lecture Notes in Comput.
  Sci.}, pages 149--161. Springer, Berlin, 2006.

\bibitem{Kucera:86}
A.~Ku{\v{c}}era.
\newblock An alternative, priority-free, solution to {P}ost's problem.
\newblock In {\em Mathematical foundations of computer science, 1986
  (Bratislava, 1986)}, volume 233 of {\em Lecture Notes in Comput. Sci.}, pages
  493--500. Springer, Berlin, 1986.

\bibitem{Kucera.Slaman:09}
A.~Ku{\v{c}}era and T.~Slaman.
\newblock Low upper bounds of ideals.
\newblock {\em J. Symbolic Logic}, {74}:{517--534}, 2009.

\bibitem{Kucera.Terwijn:99}
A.~Ku{\v{c}}era and S.~Terwijn.
\newblock Lowness for the class of random sets.
\newblock {\em J. Symbolic Logic}, 64:1396--1402, 1999.

\bibitem{Kucera.Nies:11}
A.~Ku\v{c}era and A.~Nies.
\newblock Demuth randomness and computational complexity.
\newblock {\em Ann. Pure Appl. Logic}, 162:504--513, 2011.

\bibitem{Melnikov.Nies:12}
A.~Melnikov and A.~Nies.
\newblock {$K$}-triviality in computable metric spaces.
\newblock {\em Proc. Amer. Math. Soc.}, 141(8):2885--2899, 2013.

\bibitem{Muenster}
A.~Nies.
\newblock Reals which compute little.
\newblock In {\em Logic Colloquium '02}, Lecture Notes in Logic, pages
  260--274. Springer--Verlag, 2002.

\bibitem{Nies:AM}
A.\ Nies.
\newblock Lowness properties and randomness.
\newblock {\em Adv. in Math.}, 197:274--305, 2005.

\bibitem{Nies:book}
A.~Nies.
\newblock {\em Computability and randomness}, volume~51 of {\em Oxford Logic
  Guides}.
\newblock Oxford University Press, Oxford, 2009.
\newblock 444 pages. Paperback version 2011.

\bibitem{Nies:ICM}
A.\ Nies.
\newblock Interactions of computability and randomness.
\newblock In {\em Proceedings of the International Congress of Mathematicians},
  pages 30--57. World Scientific, 2010.

\bibitem{Soare:87}
Robert~I. Soare.
\newblock {\em Recursively Enumerable Sets and Degrees}.
\newblock Perspectives in Mathematical Logic, Omega Series. Springer--Verlag,
  Heidelberg, 1987.

\bibitem{Solovay:75}
R.~Solovay.
\newblock Handwritten manuscript related to {C}haitin's work.
\newblock IBM Thomas J. Watson Research Center, Yorktown Heights, NY, 215
  pages, 1975.

\end{thebibliography}
%
% % 

\end{document}